\documentclass[11pt]{amsart}
\usepackage[margin=1.0in]{geometry}

\usepackage[T1]{fontenc}
\usepackage[utf8]{inputenc}
\usepackage{lmodern}
\usepackage{hyperref, amsmath, amssymb, mathtools}
\usepackage{verbatim}
\usepackage{amsthm}
\usepackage{amsfonts}
\usepackage{amsmath}
\usepackage{mathrsfs}  
\usepackage{comment}
\usepackage{chngcntr}
\usepackage{blkarray}
\usepackage{tikz-cd}
\usepackage{ctable}

\newtheorem{thm}{Theorem}
\newtheorem*{thm*}{Theorem}
\newtheorem{lemma}[thm]{Lemma}
\newtheorem{corollary}[thm]{Corollary}
\newtheorem{conjecture}[thm]{Conjecture}
\newtheorem{prop}[thm]{Proposition}
\newtheorem*{prop*}{Proposition}

\theoremstyle{definition}
\newtheorem{defn}[thm]{Definition}
\theoremstyle{remark}
\newtheorem{remark}[thm]{Remark}
\newtheorem{example}[thm]{Example}
\newtheorem*{example*}{Example}
\numberwithin{thm}{section}
\numberwithin{equation}{section}

\usepackage{xcolor}

\newcommand{\R}{\mathbb{R}}
\newcommand{\Q}{\mathbb{Q}}
\newcommand{\Z}{\mathbb{Z}}
\newcommand{\F}{\mathbb{F}}

\newcommand{\A}{\mathbb{A}}
\newcommand{\GL}{\text{GL}}

\newcommand{\calA}{\mathcal{A}}
\newcommand{\calF}{\mathcal{F}}

\newcommand{\calM}{\mathcal{M}}
\newcommand{\calC}{\mathcal{C}}

\newcommand{\calR}{\mathcal{R}}
\newcommand{\calH}{\mathcal{H}}
\newcommand{\calU}{\mathcal{U}}
\newcommand{\calO}{\mathcal{O}}
\newcommand{\calN}{\mathcal{N}}

\newcommand{\bP}{\mathbb{P}}
\newcommand{\bx}{\mathbf{x}}

\newcommand{\bd}{\mathbf{d}}
\newcommand{\ba}{\mathbf{a}}
\newcommand{\bm}{\mathbf{m}}
\newcommand{\br}{\mathbf{r}}
\newcommand{\bc}{\mathbf{c}}
\newcommand{\bD}{\mathbf{D}}

\newcommand{\fp}{\mathfrak{p}}

\newcommand{\fa}{\mathfrak{a}}

\newcommand{\Pic}{\text{Pic}}
\newcommand{\rk}{\text{rk}}
\newcommand{\Ind}{\text{Ind}}

\newcommand{\ord}{\text{ord}}
\newcommand{\SL}{\text{SL}}

\newcommand{\Aut}{\text{Aut}}
\newcommand{\Gal}{\text{Gal}}

\title{Counting points on a family of degree one del Pezzo surfaces}
\author[K. Woo]{Katharine Woo}
\address{Department of Mathematics, Stanford University, Stanford, CA 94305}
\email{katywoo@stanford.edu}
\date{\today}

\begin{document}

\maketitle

\begin{abstract}

    We study the rational points on the elliptic surface given by the equation:
    $$y^2 = x^3 + AxQ(u,v)^2 + BQ(u,v)^3,$$
    where $A,B\in \mathbb{Z}$ satisfy that $4A^3-27B^2\neq 0$ and $Q(u,v)$ is a positive-definite quadratic form. We prove asymptotics for a special subset of the rational points of increasing height, specifically those that are integral with respect to the singularity. This method utilizes Mordell's parameterization of integral points on quadratic twists on elliptic curves, which is based on a syzygy for invariants of binary quartic forms. 
    
    Let $\mathcal{F}(A,B)$ denote the set of binary quartic forms with invariants $-4A$ and $-4B$ under the action of $\textrm{SL}_2(\Z)$. We reduce the point-counting problem to the question of determining an asymptotic formula for the correlation sums of representation numbers of binary quadratic and binary quartic forms, where the quartic forms range in $\mathcal{F}(A,B)$. These sums are then treated using a connection to modular forms. 
\end{abstract}

\section{Introduction}\label{sec: intro}
In 2002, Swinnerton-Dyer \cite{Swinnerton-Dyer} wrote that 
\begin{quote}
    It is known that Del Pezzo surfaces of degree $d > 4$ satisfy the Hasse principle and weak approximation; indeed those of degree 5 or 7 necessarily contain rational points. Del Pezzo surfaces of degree 2 or 1 have no aesthetic merits and have attracted little attention; it seems sensible to ignore them until the problems coming from those of degrees 4 and 3 have been solved.
\end{quote}
In this paper, we will attempt the non-sensible task of studying the rational and integral points on certain singular del Pezzo surfaces of degree one; specifically, our surfaces will be elliptic surfaces that parameterize certain families of quadratic twists of elliptic curves. 

Let $A,B\in \Z$ be integers satisfying that $4A^3-27B^2 \neq 0$ and let $Q(u,v)$ be a positive-definite binary quadratic form. We will study the singular del Pezzo surfaces of degree one defined by the equation: 
$$S_{A,B,Q}:= \{y^2= x^3 + AxQ(u,v)^2 + BQ(u,v)^3\}.$$
These surfaces live in weighted projective space $\bP(2,3,1,1)$; thus, the rational points are the integral points $(x,y,u,v) \in S_{A,B,Q}(\Z)$ satisfying that $\gcd(x,u,v) = 1$ (which implies that $\gcd(x,y,u,v)=1$ for this surface.)

We define the height function on $S_{A,B,Q}$ as the restriction of the natural height on $\bP(2,3,1,1):$ $$\tilde{h}((x,y,u,v)) = \max(|x|^{1/2},|y|^{1/3},|u|,|v|).$$
The following refinement of Manin's conjecture specialized for the surface $S_{A,B,Q}$ predicts the count for rational points on $S_{A,B,Q}$ over the fixed number field $\Q$; we present the version given in \cite[Conjecture 2.3]{Browning-book}: 
\begin{conjecture}[Batyrev-Manin {\cite[Conjecture C']{BatyrevManin}}]\label{conj: Manin}
Define the open sets $U_1 = \bP(2,3,1,1)\backslash \{y=0\}$ and $U_2 = \bP(2,3,1,1) \backslash \{Q(u,v)=0\}.$ Define the point-counting function $$\Tilde{N}(T) := \#\{(x,y,u,v) \in S_{A,B,Q}(\Z)\cap U_1\cap U_2: \tilde{h}((x,y,u,v))\leq T, \gcd(x,u,v)=1\}.$$ Assume that $S_{A,B,Q}(\Q)\neq \emptyset.$
Then there exists a constant $C_{A,B,Q}$ such that $$\Tilde{N}(T) \sim C_{A,B,Q}T\log(T)^{\varrho_{A,B,Q} -1},$$
where $\varrho$ is the rank of the Picard group over $\Q$ of the desingularization $\Tilde{S}_{A,B,Q}$ of $S_{A,B,Q}$. 
\end{conjecture}

 This constant $C_{A,B,Q}$, known as Peyre's constant \cite{PeyreConstant}, can be expressed in terms of a product of local densities, the Brauer group, and the volume of a certain polytope in $\Pic(\tilde{S}_{A,B,Q})_{\R};$ we will not focus on this constant in this paper. Note that $C_{A,B,Q}$ is not guaranteed to be nonzero in the above prediction. 

 The exponent in the prediction can be explicitly calculated as well for $S_{A,B,Q}$. Let $\Delta(Q)$ be the discriminant of $Q(u,v)$ and $K_Q = \Q(\sqrt{\Delta(Q)}) = \Q[x]/(Q(x,1))$. In \S\ref{sec: Picard}, we compute that the Picard rank $\varrho_{A,B,Q}$ satisfies that \begin{equation}\label{eq: pic rank}\varrho_{A,B,Q} = \begin{cases}
    5, & z^3 - A\Delta(Q) z + B\Delta(Q)^{3/2} \text{ splits completely over }K_Q, \\
    4, & z^3 - A\Delta(Q) z + B\Delta(Q)^{3/2} \text{ splits into a linear and quadratic factor over }K_Q,\\
    3, & \text{otherwise. }
\end{cases}\end{equation}
Hence, for a generic $A,B$ and $Q(u,v)$, the conjecture asserts $$\Tilde{N}(T) \sim c_{A,B,Q} T\log(T)^2.$$

While asymptotics for Manin's conjecture have not been established for any del Pezzo surface of degree one, both upper and lower bounds have been explored. Munshi uses sieve methods to investigate in \cite{Munshi-dp1-bound} singular surfaces of the form $\{y^2 = (x+\lambda Q(u,v))(x+\delta Q(u,v))(x+\overline{\delta}Q(u,v))\}$, where $\lambda\in \Z$, $\delta$ generates an imaginary quadratic field, and $Q(u,v)$ is a positive-definite quadratic form. He bounds the number of rational points of height up to $T$ by $\ll_\epsilon T^{5/4+\epsilon}$. In \cite{LeBoudec-dp1-bound}, Le Boudec improves this bound to $\ll_\epsilon T^{1+\epsilon}$ for surfaces of the form $\{y^2=(x-e_1Q(u,v))(x-e_2Q(u,v))(x-e_3Q(u,v))\}$, where $e_1,e_2,e_3\in \Z$ are distinct. 

For nonsingular del Pezzo surfaces of degree one, which are elliptic surfaces, one can use the bounds derived by Bhargava, Shankar, Taniguchi, Thorne, Tsimerman, and Zhao \cite{ECrankBhargavaetal} and Helfgott and Venkatesh \cite{ECrankHelfgottVenkatesh} for integral points on elliptic curves to establish nontrivial upper bounds of size $O(T^{2.87+\epsilon})$. Additionally, motivated by an observation of Heath-Brown in \cite{HBCubic}, Bonolis and Browning explain in \cite{BonolisBrowning} how to achieve an upper bound of $O_\epsilon(T^{2+\epsilon})$ using a suitable rank growth hypothesis for elliptic curves. For lower bounds, Frei, Loughran, and Sofos \cite{FreiLoughranSofos} established a lower bound of the correct order of magnitude for del Pezzo surfaces of degree one after passing to an extension of $\Q$ of degree at most 138240. 

The lowest degree del Pezzo surface for which Manin's conjecture is known is a singular del Pezzo surface of degree two studied by Baier and Browning in 2013 \cite{Baier-Browning-delPezzodegree2}. For higher degrees, more is known. In \cite{delaBrowning-Manin-degree-4-2} and \cite{delaBrowning-Manin-degree-4}, de la Bret\`eche and Browning derive Manin asymptotics for certain del Pezzo surfaces of degree four. 
For more known cases of Manin's conjecture, especially for higher degrees, we refer to \cite{Browning-Survey} and \cite{Browning-book}.

\medskip

For our surface, we will count a specific subset of the rational points. 
\begin{defn}
    A rational point $(x,y,u,v) \in S_{A,B,Q}(\Z)$ is called \textbf{integral with respect to the singularity} if $\gcd(x,y,Q(u,v))=1$. From the structure of $S_{A,B,Q}$, the condition $\gcd(x,y,Q(u,v))=1$ is equivalent to $\gcd(x,Q(u,v))=1$.  \footnote{Another interpretation of this $\gcd$ condition is that $(x,y,u,v)$ is integral with respect to the singularity if and only if $(x,y,u,v)\in S_{A,B,Q}(\Z/p^k\Z)$ is nonsingular for every prime power $p^k$.}
\end{defn}
To see that this gcd condition is connected to the singularity of $S_{A,B,Q}$, we recall that the singularity of $S_{A,B,Q}$ is given by $x=y=Q(u,v)=0$. By viewing the singularity as a boundary divisor, we can view these points in the framework developed by Derenthal and Wilsch \cite{DerenthalWilsch} and by Chambert-Loir and Tschinkel \cite{ChambertLoirTschinkel} for rational points that are integral with respect to boundary divisors. In both of the aforementioned papers, and the later work of Santens \cite{Santens}, conjectures are formulated for counting integral points with respect to boundary divisors under certain geometric assumptions (for example, split log-Fano varieties) that $S_{A,B,Q}\backslash\{x=y=Q(u,v)=0\}$ does not satisfy. 
Furthermore, it is both expected, and verified through our parameterization of the integral points, that the set of integral points with respect to the singularity is contained in a thin set. It is also of note that in many formulations of Manin's conjecture for integral points one excludes a thin set, so it is unclear what the expected asymptotic should be for these points. We, however, expect that a modification of this method will produce a lower bound towards Manin's conjecture, unless there are obstructions to rational points on certain fibers; this will be explained further in Remark \ref{rem: all rational points}. \\

We must further restrict which integers $A$, $B$, and positive-definite quadratic forms $Q(u,v)$ we study; this assumption comes from the automorphic techniques in our method for counting points. Let $\mathcal{F}(A,B)$ denote the finite set of binary quartic forms with invariants $-4A$ and $-4B$, up to equivalence by $\SL_2(\Z)$. 

\begin{defn} 
    Let $A,B\in \Z$ and $Q(u,v)$ be a positive-quadratic form. The tuple $(A,B,Q(u,v))$ is called \textbf{disassociated} if the quadratic extension determined by $Q(u,v),$ denoted as $K_Q$ is not contained in the splitting field of $F$, for all $F\in \mathcal{F}(A,B).$
\end{defn}
 Finally, for simplicity, we define a slightly different height function: $$h((x,y,u,v)) = \max(|x|^{1/2}, |y|^{1/3}, |Q(u,v)|^{1/2}). $$This height function is within a scalar multiple of the natural height function $\tilde{h}((x,y,u,v))$ coming from weighted projective space. 
\begin{thm}\label{thm: Manin}
    Let $A,B\in \Z$ such that $4A^3-27B^2 \neq 0$ and $Q(u,v)$ be a positive-definite quadratic form. Assume that $(A,B,Q(u,v))$ is disassociated. Let $U_1 = \bP(2,3,1,1)\backslash\{y=0\}$ and $U_2 = \bP(2,3,1,1)\backslash \{Q(u,v)=0\}.$ Define the point-counting function $$N(T):= \#\{(x,y,u,v)\in S_{A,B,Q}(\Z)\cap U_1\cap U_2: h((x,y,u,v))\leq T, \gcd(x,Q(u,v))=1\}.$$
    Then the following asymptotic holds as $T\rightarrow\infty$: $$N(T) = c_{A,B,Q}T+O_{A,B,Q}(T\log(T)^{-10^{-8}}).$$
\end{thm}

\begin{remark}
The assumption that $(A,B,Q(u,v))$ is disassociated is stronger than necessary for our methods to produce as asymptotic for $N(T)$. Let $\calH$ denote the Hilbert class field of $K_Q$; let $K_F$ denote the splitting field of a binary form $F$. If we have that either $K_Q\not\subset K_F$ or $K_Q\subset K_F$ and $\calH\cap K_F = K_Q$, then there exists an explicit exponent $0\leq \beta_{A,B,Q}\leq 2$ such that 
$$N(T) = c_{A,B,Q}T\log(T)^{\beta_{A,B,Q}}+O_{A,B,Q}(T\log(T)^{\beta_{A,B,Q}-10^{-8}}).$$
    Moreover, for a given $A,B$ and $Q(u,v)$, this exponent $\beta_{A,B,Q}$ is computable by combining the explicit parameterization given in \S\ref{sec: syzygy} of the integral points and utilizing Mordell's fundamental domain for binary quartics with given invariants (see \cite{Mordell} for explicit bounds). 
\end{remark}

\begin{remark}
    In this paper, we do not give a criteria for when $c_{A,B,Q}$ is nonzero -- in \S\ref{sec: ex and lower bounds} we look at specific examples of surfaces and their respective leading constants. In all cases, lower bounds for Manin's conjecture are established. In \S\ref{subsec: BM example}, we also give an example of a surface when a Brauer-Manin obstruction completely obstructs integral points in one of, but not all of, the fibers -- this phenomenon must be considered in any eventual evaluation of the constant $c_{A,B,Q}.$ 
\end{remark}

The new tool employed in establishing Theorem \ref{thm: Manin} is the use of syzygies. The rational points of $S_{A,B,Q}$ that are integral with respect to the singularity can be parameterized using a syzygy of Cayley \cite{Cayley} and Hermite \cite{Hermite} for binary quartic forms; this syzygy specifically connects binary quartics to the quadratic twists of elliptic curves. This crucial observation was initially made by Mordell \cite{Mordell} during his seminal investigation of the rational points of elliptic curves. Recently, Duke \cite{Duke} provided the quantitative version of this relationship that we will utilize in \S\ref{sec: syzygy}. A loose summary of the parameterization is such: the solutions to the equation $$\{(x,y,n): y^2 = x^3 + Axn^2 + Bn^3, \gcd(x,n)=1\}$$
are given by 
$$\left\{(H_F(\bm), T_F(\bm), F(\bm)): 
     \bm\in \Z^2,
     F\in \mathcal{F}(A,B)
\right\},$$
where $H_F(m,n)$ (resp. $T_F(m,n)$) is a binary quartic (resp. sextic) that is formed with the coefficients of $F$. This parameterization will be further elucidated in \S\ref{sec: syzygy}. We note that this parameterization makes the set of rational points that are integral with respect to the singularity a thin set of type II (as defined by Serre \cite[\S9.7]{SerreBook}); however, we expect that the set of rational points counted in Conjecture \ref{conj: Manin} to not be thin.\\

Our key analytic input will be an asymptotic for correlation sums between binary quartic and binary quadratic forms. First, let $Q(u,v)$ be a positive-definite integral quadratic form. We define $$r_Q(n):= \#\{u,v\in \Z: Q(u,v)=n\}.$$
We want to study the correlations of $r_Q(n)$ with the number of solutions of a binary form. Based on the author's proof of a similar correlation sum in \cite{MyManinChatelet}, we will prove the following theorem: 
\begin{thm}\label{thm: correlation}
    Let $F(x,y)$ be a squarefree binary form of degree $\leq 4$ and $Q(u,v)$ a positive-definite integral quadratic form. Assume that the splitting field of $F(x,y)$, denoted as $K_F$, and the Hilbert class field of $K_Q$ are independent. Let $\calR \subset \R^2$ be a fixed compact subset. Let $\alpha \bmod M$ be a fixed residue class in $(\Z/M\Z)^2$. Then there  exists a constant $c_{\alpha, F,Q}$ such that $$\sum_{\substack{x,y\in X\calR \\ \gcd(x,y)=1\\ (x,y) \equiv \alpha \bmod M}} r_Q(F(x,y)) = c_{\alpha, F,Q} X^2|\calR| \log(X)^{\beta_{F,Q}}+O_{\calR,F,Q,M}(X^2(\log(X))^{\beta_{F,Q}-10^{-8}}),$$
   where we have that $$\beta_{F,Q} = \begin{cases}
        2, & F(x,y) = Q_1(x,y) Q_2(x,y) \text{ where $Q_1,Q_2$ are quadratic and }K_{Q_1}=K_{Q_2}=K_{Q}, \\
        1, & F(x,y) = Q_1(x,y)Q_2(x,y) \text{ where $Q_1,Q_2$ are quadratic and }K_{Q_1}=K_{Q}\neq K_{Q_2},\\
        1, & F(x,y) \text{ is irreducible and }K_Q\subset K_F, \\
        0, & \text{otherwise.}
    \end{cases}.$$
\end{thm}
\begin{remark}
    Observe that for Theorem \ref{thm: Manin}, by the assumption that $(A,B,Q(u,v))$ are disassociated, all of the relevant exponents $\beta_{F,Q}$ will be zero. 
\end{remark}

 For $G(u,v):=u^2+v^2$, there are many results in the literature covering such correlation sums. Heath-Brown \cite{HB-linear}, and later de la Bret\`eche and Browning \cite{delaBrowning-linear}, analyzed and proved asymptotics when $F(x,y)$ factors into four distinct linear forms $L_1,L_2,L_3,L_4$:
$$\sum_{x,y\leq X} r_G(L_1(x,y))r_G(L_2(x,y))r_G(L_3(x,y))r_G(L_4(x,y)) = c_{L_i,Q} X^2 + O(X^2\log(X)^{-\delta}).$$ Note that using different methods, Matthiesen \cite{Matthiesen} also established the above estimate for $F(x,y)$ a product of arbitrarily many distinct linear factors. Later, in \cite{delaBrowning-quad}, de la Bret\`eche and Browning, and later Destagnol in \cite{Destagnol}, resolved the case when $F(x,y)$ splits into two linear factors and an irreducible quadratic. Finally, in \cite{delaTenenbaum-Manin}, de la Bret\`eche and Tenenbaum solved the hardest case: when $F(x,y)$ is an irreducible binary quartic form. We remark that all of the techniques used in the aforementioned results work for any positive-definite quadratic form of class number one, and all but \cite{Matthiesen} build upon the seminar work of Daniel \cite{Daniel} for divisor functions along binary quartics. 

In \S\ref{sec: correlations}, we rework the tools from \cite{MyManinChatelet} for $Q(u,v)= u^2+\Delta v^2$ to handle an arbitrary positive-definite form and establish Theorem \ref{thm: correlation}. We remark that in combination with the ideas of \cite{MyManinChatelet} for $\Delta<0$, it is possible to replace $Q(u,v)$ with an arbitrary indefinite binary quadratic forms (by adding in the use of argument-detecting Hecke Grossencharacters and restricting to short-ish intervals). For the sake of clarity and simplicity of notation, we do not work out the indefinite case in this paper. This section also uses the work in the \cite{MyBaseChange} on correlation sums of cusp form coefficients along polynomial values and the connection to base change for $\GL(2)$.

\section*{Acknowledgements}
The author would like to thank William Duke for proposing the problem, her advisor Peter Sarnak for his guidance and support, Tim Browning for many enlightening conversations about Manin's conjecture and his helpful feedback, and Dante Bonolis and Florian Wilsch for discussions around integral points with respect to singularities. The author would also like to thank Niven Achenjang, Yijie Dao, Ulrich Derenthal, Amit Ghosh, Trajan Hammonds, Dan Loughran, and Shuntaro Yamagishi for comments on earlier versions of the proof. 

This material is based upon work supported by the National Science Foundation Graduate Research Fellowship Program under Grant No. DGE-2039656. Any opinions, findings, and conclusions or recommendations expressed in this material are those of the author(s) and do not necessarily reflect the views of the National Science Foundation.


\section{Syzygys of binary quartic forms}\label{sec: syzygy}
First, we recall some facts about binary quartics. Then, we will discuss the reduction of Theorem \ref{thm: Manin} to Theorem \ref{thm: correlation} via this parameterization given by syzygys of binary quartic forms. 

\subsection{Binary quartic forms and Mordell's parameterization}\label{subsec: Mordell}
Let us write $$F(x,y) = ax^4+4bx^3y+6cx^2y^2+4dxy^3+ey^4.$$
The space of integral binary quartics $F(x,y)$ can be parameterized by the coefficient variables $a,b,c,d,e.$ 
There is an action of $\SL_2(\Z)$ on the space of integral binary quartics: $$\begin{pmatrix}
    \alpha &\beta \\\gamma &\delta 
\end{pmatrix}\cdot F(x,y) = F(\alpha x+ \beta y, \gamma x+ \delta y).$$
This action then induces an action of $\SL_2(\Z)$ on $\Z[a,b,c,d,e]$, and we can ask for the generators of the ring of invariants: 
$$\Z[a,b,c,d,e]^{\SL_2(\Z)} = \Z[I,J].$$
Here, $I(a,b,c,d,e)$ and $J(a,b,c,d,e)$ are defined as follows:
$$I = ae-4bd+3c^2$$ $$J = ace+2bcd-b^2e -d^2a -c^3.$$
For $F(x,y) = ax^4+4bx^3y+6cx^2y^2+4dxy^3+ey^4$, we will also denote these expressions as $I(F)$ and $J(F)$. 

\begin{example}
    The discriminant polynomial $\Delta(F)$ of a binary quartic form $F(x,y)$ will also be invariant under the action of $\SL_2(\Z)$; so, $\Delta(F) \in \Z[a,b,c,d,e]^{\SL_2(\Z)}$ should have an expression in terms of $I(F)$ and $J(F)$. This is indeed true and the formula is reminscient of the discriminant of an elliptic curve: $$\Delta(F) = I(F)^3 - 27J(F)^2.$$
\end{example}

Next, in order to write down the syzygy between $I(F)$ and $J(F)$, we must define two covariant functions. 
\begin{defn}
    The \textbf{Hessian} of a binary quartic form $F(x,y)$ is the binary quartic form\footnote{Here, $F_x = \frac{\partial F(x,y)}{\partial x}$ and $F_{xy} = \frac{\partial^2 F(x,y)}{\partial x\partial y}$.}: $$H_F(x,y) = \frac{1}{44}\det\begin{pmatrix}
        F_{xx} & F_{xy}\\ F_{yx} & F_{yy}
    \end{pmatrix}.$$
    Given a binary quartic form $F(x,y)$, we call the following function the  \textbf{Jacobian covariant polynomial}: $$T_F(x,y) = \frac{-1}{8} \det\begin{pmatrix}
        F_{x}& F_{y}\\
        H_{F,x} & H_{F,y}
    \end{pmatrix}.$$
    This polynomial $T_F(x,y)$ will be a binary sextic form.
\end{defn}

\begin{lemma}\label{lem: coprime Hessian and quartic}
    If $F(x,y)$ is a squarefree binary quartic form, then $F(x,y)$ and $H_F(x,y)$ will share no common factors over $\overline{\Q}.$
\end{lemma}
The above lemma is proved via a computation, which will be shown in \S\ref{subsec: Mordell proof}.
Finally, we introduce our fundamental syzygy, which is due independently to Cayley \cite{Cayley} and Hermite \cite{Hermite}, that will form the basis of Mordell's parameterization. 
\begin{thm}[Cayley \cite{Cayley}, Hermite \cite{Hermite}]\label{thm: Cayley Hermite syzygy}
    For $F(x,y)$ a binary quartic form, the following relation holds: 
    $$T_F(x,y)^2 = -4H_F(x,y)^3+I(F) \cdot H_F(x,y) F(x,y)^2 - J(F)\cdot  F(x,y)^3.$$
\end{thm}
\medskip

In \cite{Mordell}, Mordell noted that the syzygy above takes the same shape as the quadratic twists of an elliptic curve. In particular, by setting $y'=T_F(x,y)/2$, $x'=-H_F(x,y)$, and $n=F(x,y)$, we get the elliptic curve $$y'^2 = x'^3 - \frac{I(F)}{4} x'n^2 - \frac{J(F)}{4} n^3.$$
From this observation, one can parameterize the solutions of the equation $y'^2 = x'^3 + Ax'n^2 + Bn^3$ using binary quartics. This parameterization first appears in a paper of Mordell \cite{Mordell}; we will introduce the quantitative form by Duke in \cite{Duke}.  

Define $\Gamma_\infty\subset \SL_2(\Z)$ to be the group that fixes the point at $\infty$: $$\Gamma_\infty:=\left\{\begin{pmatrix}
    1 & x \\ 0 & 1 
\end{pmatrix}: x\in \Z\right\}.$$
We also need the notion of an admissible binary quartic form. 
\begin{defn}
    A binary quartic form $F(x,y) = ax^4+4bx^3y+6cx^2y^2 + 4d xy^3 + ey^4$ is \textbf{admissible} if $\gcd(a,b)=1.$
\end{defn}
Finally, we define the space of binary quartic forms with invariants $-4A$ and $-4B:$ 
\begin{multline*}\mathscr{F}(A,B) := \{\text{admissible integral binary quartic forms }F(x,y): \\ I(F) = -4A, J(F)=-4B, \text{ and }F(1,0)>0\}.\end{multline*}
We note that this set is invariant under the action of $\Gamma_\infty$. In fact, without the conditions of admissibility and that $F(1,0)>0$, this space would be invariant under the action of $\SL_2(\Z)$. In Theorem 1.7 of Bhargava and Shankar \cite{BhargavaShankar}, the congruences necessary on $A$ and $B$ so that this set is nonempty are listed.

In the following result, the connection between the set $\mathscr{F}(A,B)$ and the integral points on the affine elliptic surface $y^2=x^3+Axn^2+Bn^3$ (the surface consisting of the quadratic twists of an elliptic curve) is made explicit. 

\begin{thm}[Duke, \protect{\cite[Proposition 2]{Duke}}] \label{thm: Duke parameterization}
Let $A,B,\in \Z$ satisfy that $4A^3-27B^2\neq 0$. Then consider the map 
    $$ \mathscr{F}(A,B)/\Gamma_\infty \rightarrow \{(x,y,n)\in \Z^2\times \Z_+: y^2 = x^3+Axn^2+Bn^3, \gcd(x,n)=1\}$$
    $$F \xrightarrow{\varphi} (-H_F(1,0), \frac{1}{2} T_F(1,0), F(1,0)).$$
    Moreover, this map will be a bijection.
\end{thm}
The above result gives us a nice expression for the following counting function for the integer points on the elliptic curve of height at most $T$: 
\begin{equation}\label{eq: counting function def}
    v(n,T):= \#\left\{(x,y)\in \Z^2: \begin{array}{c }y^2=x^3+Axn^2+Bn^3,\\  \gcd(x,n)=1, |x|^{1/2},|y|^{1/3}\leq T,  y\neq 0\end{array}\right\}.
\end{equation}
\begin{corollary}
    Let $A,B\in \Z^2$ satisfy that $4A^3-27B^2\neq 0$. Then we have that $$v(n,T) = \sum_{\substack{F\in \mathscr{F}(A,B)/\Gamma_\infty\\ F(1,0)=n}} \mathbf{1}_{T_F(1,0)\neq 0} \mathbf{1}_{|H_F(1,0)|^{1/2},|T_F(1,0)/2|^{1/3}\leq T}.$$
\end{corollary}
However, this sum is still not ideal for our application, as it sums over an infinite set $\mathscr{F}(A,B)/\Gamma_\infty$. Instead, we look to the the set of representatives under the action of $\SL_2(\Z)$: $$\calF(A,B) := \{\text{integral binary quartic forms }F(x,y): I(F) = -4A, J(F)=-4B\}/\SL_2(\Z).$$
Mordell \cite{Mordell} showed that the number of equivalence classes $\#\calF(A,B)$ is finite. 
Morally, we want to use the following idea: \begin{equation}\label{eq: moral relation between scrF and calF}\calF(A,B) \times \SL_2(\Z)/\Gamma_\infty \approx \mathscr{F}(A,B) /\Gamma_\infty.\end{equation}
However, the above relation is not quite true; we need to make appropriate adjustments to handle our restriction to admissible forms and the condition that $F(1,0)>0$. 

\begin{defn}
    Fix a binary quartic form $F$. A pair $(m_1,m_2)\in \Z^2$ is \textbf{admissible for $F$} if $$\begin{pmatrix}
        m_1 & *\\ m_2 & *
    \end{pmatrix}\cdot F$$ is an admissible binary quartic form.
\end{defn}
We note that in order for $(m_1,m_2)\in \Z^2$ to be admissible for a fixed binary quartic $F$, it must satisfy that $\gcd(m_1,m_2)=1.$ In fact, it is shown in \cite[Proposition 4]{Duke} that $(m_1,m_2)$ are admissible for $F$ if and only if the pair $(m_1,m_2)$ satisfies $$\gcd(m_1,m_2)=1, \text{ and }\forall p\mid \Delta(F) \text{ prime}, (F(m_1,m_2),H_F(m_1,m_2)) \neq (0,0)\bmod p.$$
We denote the $q = \prod_{p\mid \Delta(F)} p$. Let us denote by $R_F$ the following congruence conditions:
\begin{equation}\label{eq: def of local condition R_F} R_F := \{ \bm \bmod q: \forall p\mid q \text{ prime}, (F(\bm),H_F(\bm))\neq (0,0)\bmod p\}.\end{equation}
We can now make observation (\ref{eq: moral relation between scrF and calF}) rigorous, as was done in \cite[Proposition 4]{Duke}; we repeat Duke's proof in \S\ref{subsec: Mordell proof}.

\begin{lemma}\label{lem: expression for v(n,T)}
    Let $A,B\in \Z^2$ satisfy that $4A^3-27B^2 \neq 0$. Define $\Aut(F)\subset \SL_2(\Z)$ as the automorphism group of $F$. Then we have that $$v(n,T) = \sum_{F\in \calF(A,B)}\frac{1}{\#\Aut(F)} \sum_{\substack{\bm\in \Z^2 \\ \gcd(m_1,m_2)=1\\ \overline{\bm}\in R_F \\ F(\bm)=n}} \mathbf{1}_{T_F(\bm)\neq 0} \mathbf{1}_{|H_F(\bm)|^{1/2},|T_F(\bm)|^{1/3}\leq T}$$
\end{lemma}
\begin{remark}
    Note that $\Aut(F)$ is finite for any element $F\in \calF(A,B).$
\end{remark}

\subsection{Reducing Theorem \ref{thm: Manin} to Theorem \ref{thm: correlation}}
Let $A,B\in \Z$ satisfy that $4A^3-27B^2\neq 0$. Then we wish to count $$N(T) = \#\{(x,y,u,v)\in S_{A,B,Q}(\Z): h((x,y,u,v))\leq T, \gcd(x,Q(u,v))=1, y\neq 0\}.$$ We can rewrite this in terms of the function $v(n,T)$ as $$N(T) = \sum_{n\leq T^2} v(n,T)r_Q(n),$$
where $r_Q(n) = \#\{u,v\in \Z: Q(u,v)=n\}$. 
Plugging in Lemma \ref{lem: expression for v(n,T)}, the expression above can be written as 
\begin{align*}N(T) &= \sum_{F\in \calF(A,B)} \frac{1}{\#\Aut(F)} \sum_{n\leq T^2} r_Q(n)\sum_{\substack{\bm\in \calA_F \\ T_F(\bm)\neq 0\\ |H_F(\bm)|^{1/2},|T_F(\bm)/2|^{1/3}\leq T}}\textbf{1}_{F(\bm) = n} \\
&= \sum_{F\in \calF(A,B)} \frac{1}{\#\Aut(F)} \sum_{\substack{\bm\in \calA_F\\ 
T_F(\bm)\neq 0\\ |H_F(\bm)|^{1/2},|T_F(\bm)/2|^{1/3}, |F(\bm)|^{1/2}\leq T}} r_Q(F(\bm)).
\end{align*}

Next, we note that since $F\in \calF(A,B)$, the discriminant $\Delta(F) = 16(4A^3-27B^2)\neq 0$. Thus, we automatically have that our sum is over squarefree binary quartic forms. Applying Lemma \ref{lem: coprime Hessian and quartic} and using that $F,H_F,T_F$ are all homogeneous, we can see that the conditions $$\{\bm\in \R^2:|H_F(\bm)|^{1/2},|T_F(\bm)/2|^{1/3}, |F(\bm)|^{1/2}\leq T, T_F(\bm)\neq 0\} =: T^{1/2}\calR_F$$
 where $\calR_F$ is a compact region; Lemma \ref{lem: coprime Hessian and quartic} ensures that this region will be noncuspidal. Moreover, this region will have boundary satisfying that $|\partial T^{1/2}\calR_F|\ll_F T^{1/2}.$  
 We can rewrite the sum $N(T)$ as $$N(T) = \sum_{F\in \calF(A,B)} \frac{1}{\#\Aut(F)}  \sum_{\substack{\bm\in T^{1/2}\calR_F\\ \gcd(m_1,m_2)=1 \\ \bm\in R_F}} r_Q(F(\bm)).$$
 Here we recall $R_F$ as the set of congruence conditions from (\ref{eq: def of local condition R_F}). So, we apply Theorem \ref{thm: correlation} to each of these sums individually to achieve Theorem \ref{thm: Manin} -- observe that since $(A,B,Q(u,v))$ are disassociated that $\beta_{F,Q}=0$ for all $F\in \calF(A,B)$. Note that $\#\calF(A,B)$ and $\#R_F$ are both finite constants only dependent on $A,B\in \Z$. \qed 

 \begin{remark}\label{rem: all rational points}
     To count all of the rational points instead of ones that are integral with respect to the singularity using the method of syzygies described above, one turns to estimating the weighted sums over binary quartic forms $F(x,y)$ with invariants ranging in a thin set $$\mathcal{I}_{A,B} = (I(F),J(F)) \in \{\lambda^2 A, \lambda^3 B: \lambda \in \Z\}.$$
    Understanding the class number of binary quartic forms , i.e. $\#\calF(A,B)$, as $A$ and $B$ range in $\mathcal{I}_{A,B}$ above seems a challenging problem at the moment. 

     Since we can parameterize all of the rational points using this method of syzygies by letting $I(F)$ and $J(F)$ vary in $\mathcal{I}_{A,B}$, we claim we can always achieve a lower bound towards Manin's conjecture using this method of attack, despite the fact that in Manin's conjecture a thin set might be excluded.  For any given $d$, let 
     \begin{equation*}
         S_{A,B,Q}^d(\Z):= \{(x,y,u,v)\in S_{A,B,Q}(\Z): \gcd(x,y,Q(u,v))=d\}.
     \end{equation*}
     As we range over all $d\in \mathbb{N}$, we have that 
     $$S_{A,B,Q}(\Z) = \cup_{d\in \mathbb{N}} S_{A,B,Q}^d(\Z);$$
     additionally, we know that $S_{A,B,Q}(\Z)$ gives all of the rational points on $S_{A,B,Q}$ (viewing the surface in the weighted projective space $\bP(2,3,1,1)$). 
     
     Assume that $S_{A,B,Q}(\Q)$ is not thin. Let $\mathcal{T}\subsetneq S_{A,B,Q}$ denote a thin set excluded from our count of rational points for Manin's conjecture. There must exist some $d_0$ (depending only on $A,B$, and $Q$) such that $\dim(\mathcal{T}\cap S_{A,B,Q}^{d_0}) < \dim (S_{A,B,Q}^{d_0})$.  
    After a quick rescaling of points by $d_0$, we can see that $S_{A,B,Q}^{d_0}(\Z)$ will be parameterized by those binary quartic forms with invariants $I(F) = d_0^2 A$ and $J(F) = d_0^3 B$. Additionally, by Mordell \cite{Mordell}, $\#\calF(d_0^2A,d_0^3B)$ will be finite. Thus, through the analogous parameterization, we can derive an asymptotic for the point counting function
     $$N^{d_0}(T) = \#\{(x,y,u,v)\in S_{A,B,Q}^{d_0}(\Z): h((x,y,u,v))\leq T, y\neq 0\}$$
     via Theorem \ref{thm: correlation} (with an error term that will depend on $d_0$, which only depends on $A,B$,and $Q$). On the other hand, since $\dim(\mathcal{T}\cap S_{A,B,Q}^{d_0})<\dim (S_{A,B,Q}^{d_0})$, it follows that 
     \begin{equation*}
         \#\{(x,y,u,v)\in S_{A,B,Q}^{d_0}\setminus \mathcal{T}(\Z): h((x,y,u,v))\leq T, y\neq 0\} \gg N^{d_0}(T).
     \end{equation*}
     
     By definition, we have the following inequality for counting rational points on $S_{A,B,Q}$ of height up to $T$:
     \begin{multline*}
         \#\{(x,y,u,v)\in S_{A,B,Q}\backslash \mathcal{T}(\Z): h((x,y,u,v,))\leq T, y\neq 0\} \\ \geq \#\{(x,y,u,v)\in S_{A,B,Q}^{d_0}\setminus \mathcal{T}(\Z): h((x,y,u,v))\leq T, y\neq 0\}.
     \end{multline*}
     So, if the asymptotic for $N^{d_0}(T)$ is nontrivial -- i.e. the leading constant is nonzero -- this approach of parameterizing rational points using syzygys of binary quartics and studying correlation sums of binary quadratics and quartics as in Theorem \ref{thm: correlation} will produce a lower bound towards Manin's conjecture. 
 \end{remark}

\subsection{Proofs of Lemmas from \S\ref{subsec: Mordell}}\label{subsec: Mordell proof}
The proof of Lemma \ref{lem: coprime Hessian and quartic}, which states that $F$ and $H_F$ share no common factors over $\overline{\Q}$, is computational. 
\begin{proof}[Proof of Lemma \ref{lem: coprime Hessian and quartic}]
    Assume that $F(x,y)$ and $H_F(x,y)$ share a common factor $(x+ry)$ over $\overline{\Q}.$
    We write that $F = (x+ry)C(x,y),$ where $C(x,y)$ is a cubic form. Then we can see that $$F_x = C(x,y) + (x+ry) C_x(x,y),$$ $$F_y = rC(x,y) + (x+ry) C_y(x,y).$$ Thus, we can compute that
$$F_{xx} = 2C_x(x,y) + (x+ry) C_{xx}(x,y),$$ 
$$F_{xy} = C_y(x,y) + r C_x(x,y) + (x+ry) C_{xy}(x,y),$$
$$F_{yx} = rC_x(x,y) + C_y(x,y) + (x+ry) C_{yx}(x,y),$$
$$F_{yy} = 2rC_y(x,y) + (x+ry) C_{yy}(x,y).$$

Now, by the definition of $H_F(x,y)$, we can see that $(x+ry)\mid H_F(x,y)$ over $\Q$ if and only if $$(x+ry) \mid 4rC_x(x,y) C_y(x,y) - (C_y(x,y)+rC_x(x,y))^2 = -(C_y(x,y)-rC_x(x,y))^2.$$
We expand out in terms of $C(x,y) = ax^3 + bx^2y+ cxy^2 + dy^3$: 
\begin{align*}C_y(x,y)- rC_x(x,y) &= (3dy^2 + 2cxy+ bx^2) - r(3ax^2 + 2bxy + cy^2) \\
&= -3rax^2 + b(x^2 - 2rxy) + c(2xy-ry^2) + 3dy^2.\end{align*}
Now, if $(x+ry) \mid C_y(x,y)-rC_x(x,y)$, then the above should be zero if we plug in $(r,-1).$ Hence, we get the following linear condition on $a,b,c,d$: $$-3r^3 a + 3r^2 b -3rc +3d = 0.$$

We note that solutions to the above linear condition can be parameterized as $a' = a$, $b' = b-ar$, $c' = c-b'r$, $d = c'r$, which is equivalent to $C(x,y) = (x+ry)(a'x^2+b'xy+c'y^2).$ So, if $(x+ry)$ divides the Hessian $H_F(x,y)$, we must have that it divides $C(x,y)$ and hence it forces $F(x,y)$ to have a double root. Since, we have assumed $F(x,y)$ to be squarefree, this gives us a contradiction. \end{proof}

In Lemma \ref{lem: expression for v(n,T)}, $v(n,T)$ is expressed in terms of the finite set of representative $\calF(A,B)$ by following the proof of \cite[Proposition 4]{Duke}. We repeat it here for the purposes of self-containment, as the lemma is crucial for the proof of the main theorem. 

\begin{proof}[Proof of Lemma \ref{lem: expression for v(n,T)}]
    Recall that 
    \begin{equation*}
        \SL_2(\Z)/\Gamma_\infty = \left\{ \begin{pmatrix}
        m_1 & * \\ m_2 & * 
    \end{pmatrix}: m_1,m_2\in \Z\right\}.
    \end{equation*}
    Fix a binary quartic $F(x,y)$ with invariants $I(F) = -4A$ and $J(F) = -4B$. 
    Define the set $$\calA_F := \{(m_1,m_2)\in \Z^2: (m_1,m_2) \text{ is admissible for }F\}\subset \SL_2(\Z)/\Gamma_\infty.$$ Then consider the map: 
    $$\varphi_F: \calA_F \rightarrow \{(x,y,n)\in \Z^2\times \Z_+: y^2 = x^3 +Axn^2 + Bn^3, \gcd(x,n)=1\},$$ $$\varphi_F(m_1,m_2) = (-H_F(m_1,m_2), T_F(m_1,m_2)/2, F(m_1,m_2)).$$
    Since $(m_1,m_2)$ are admissible for $F$, we know that the tuple $(-H_F(m_1,m_2),T_F(m_1,m_2), F(m_1,m_2))$ will lie in the set $\{(x,y,n): y^2 = x^3 + Axn^2 + Bn^3, \gcd(x,n)=1\}.$ 
    
    Now, the kernel of this map will have size $\#\Aut(F)$ since $\varphi_F$ is derived from the injection in Theorem \ref{thm: Duke parameterization}; in particular, $$\varphi_F = \varphi \circ \left(F\mid \begin{pmatrix}
        m_1 & * \\ m_2&* 
    \end{pmatrix}\right).$$
    
    Finally, since the map $\varphi$ from Theorem \ref{thm: Duke parameterization} was a bijection, the union of $\mathrm{Im}(\varphi_F)$, as $F$ ranges in $\calF(A,B)$, will cover all of $\{(x,y,n):y^2 = x^3 + Axn^2 + Bn^3, \gcd(x,n)=1\}.$
    Thus, we get that $$v(n,T) = \sum_{F\in \calF(A,B)} \frac{1}{\#\Aut(F)} \#\left\{\bm\in \calA_F: \begin{array}{c}F(\bm) = n, T_F(\bm)\neq 0, \\ \max(|H_F(\bm)|^{1/2},|T_F(\bm)/2|^{1/3})\leq T\end{array}\right\}.$$
    Plugging in the definition of the congruence conditions $R_F$ and rearranging the sum, we achieve the lemma. 
\end{proof}

\section{Correlations of binary quadratics and binary quartics}\label{sec: correlations}
The goal of this section is to establish Theorem \ref{thm: correlation}. This proof will follow closely the proof of \cite[Proposition 5.2]{MyManinChatelet} with $L=1$, so we aim to highlight the differences in the two proofs. There are two key changes in this formulation. First, we let $Q(u,v)$ be \textit{any positive-definite binary quadratic}, rather than of the form $x^2 + ay^2$. Second, we allow for a \textit{congruence restriction} $(x,y)\equiv \alpha \bmod M$; however, we are allowing ourselves arbitrary dependency on $M$ in the error term, which eases the difficulty of this constraint. We also point out that in this work we do not handle the indefinite case (although we expect this case to follow with the method presented in \cite{MyManinChatelet} for indefinite forms), nor do we inspect the leading constant. 

There is also a noticeable difference with the setting of \cite[Proposition 5.2]{MyManinChatelet} (the goal being Manin's conjecture for Ch\^atelet surfaces) in that we do not have the extra summation over variable labeled as $t$; this is due to counting integral points rather than rational points. Thus, the argument in \cite[Section 6]{MyManinChatelet} is not applicable for this correlation sum; correspondingly, the restrictions of $(A,B,Q(u,v))$ come from the fact that the arguments of \cite[Section 7]{MyManinChatelet} can not provide a bound for all finite image Hecke characters and are highlighted further in the \cite{MyBaseChange}.

\subsection{Notation}
\begin{itemize}
    \item $K_Q$ is the quadratic extension of $\Q$ by the binary quadratic form $Q(u,v)$; this is an imaginary quadratic field.
    \item $\chi$ is the real quadratic character such that $\zeta_{K_Q}(s) = \zeta(s) L(s,\chi)$.
    \item $\chi_q$ is the real quadratic character of modulus $q$.
    \item $\Delta$ is the discriminant of $Q(u,v)$. 
    \item $\fa, \fp$ are ideals of $\calO_{K_Q}$. 
    \item $\calC$ denotes the class group of $K_Q$. 
    \item $\calH$ denotes the Hilbert class field of $K_Q$; then $\Gal(\calH/K_Q) \cong \calC$,
    \item $K_F$ denotes the splitting field of a binary form $F(u,v).$
    \item We define the $q$-part of an integer $n$ as: $$p_q(n) = \prod_{\substack{p\textrm{ prime}\\ p^e\| q\\ p\mid q}}p^e.$$
    We also define the part of $n$ coprime to $q$ as $$p_{\neg q}(n) = n/p_q(n) = \prod_{\substack{p\textrm{ prime}\\ p^e \| n \\ p\nmid q}} p^e,$$
\end{itemize}

\subsection{Decomposition of representation number}

Let us start by decomposing the representation number $r_Q(n)$ into Eisenstein and cuspidal components using the class group. Let $\calU$ denote the unit group of $K_Q$. Then we know that 
\begin{equation*}
    r_Q(n) = |\calU| \cdot \#\{\fa\subset \calO_{K_Q}: N(\fa)=n, \fa \text{ lies in }[Q]\},
\end{equation*}
where $[Q]$ is the ideal class given by the quadratic form $Q(u,v)$. If $Q(u,v) = u^2 + a v^2$, then $[Q]=1$ is the principal class. 

Using class group characters, we can reexpress the sum above as: 
\begin{equation*}
    r_Q(n) = \frac{|\calU|}{|\calC|} \sum_{\psi\in \hat{\calC}} \overline{\psi([Q])} \sum_{N(\fa)=n} \psi(\fa).
\end{equation*}
As in \cite[Section 3.6]{MyManinChatelet}, the Fourier coefficients of the Eisenstein (denoted $\lambda_E(n))$ and cuspidal (denoted $\lambda_C(n))$ components are respectively:
\begin{equation}\label{eq: def Eisenstein Fourier}
    \lambda_E(n) := \frac{|\calU|}{|\calC|} \sum_{\substack{\psi\in \hat{\calC}\\ \ord(\psi)\leq 2}} \overline{\psi([Q])} \sum_{N(\fa)=n} \psi(\fa) = \frac{|\calU|}{|\calC|} \sum_{\substack{q_1q_2=\Delta \\ \gcd(q_1,q_2)=1}}' \psi_{q_1,q_2}([Q]) \cdot \varepsilon_{q_1,q_2}(n).
\end{equation}

\begin{equation}\label{eq: def cuspidal Fourier}
    \lambda_C(n) := \frac{|\calU|}{|\calC|} \sum_{\substack{\psi\in \hat{\calC}\\ \ord(\psi)\geq 3}} \overline{\psi([Q])} \sum_{N(\fa)=n} \psi(\fa).
\end{equation}
Above, $\psi_{q_1,q_2}$ denotes the genus character and our sum is restricted to those $(q_1,q_2)$ that generate genus characters; as described in \cite[\S12.5]{IwaniecKowalski}, all genus characters will be given by pairs $(q_1,q_2)$ satisfying that $\gcd(q_1,q_2)=1$ and $q_1q_2=\Delta$ (although not all pairs will necessarily give a genus character). Let $\chi_{q_i}$ denotes the primitive real quadratic character with modulus $q_i$. For $\gcd(n,\Delta)=1$, these Eisenstein coefficients are given by 
\begin{equation*}
    \varepsilon_{q_1,q_2}(n) = \sum_{d\mid n} \chi_{q_1}(d)\chi_{q_2}(n/d).
\end{equation*}

From our decomposition of $r_Q(n)$, we can rewrite the sum in Theorem \ref{thm: correlation} in terms of: 
\begin{equation}\label{eq: def M(q_1,q_2)}
    \calM(X;\calR, \alpha \bmod M; q_1,q_2) = \sum_{\substack{x,y\in X\calR \\ \gcd(x,y)=1\\ (x,y)\equiv \alpha \bmod M}} \varepsilon_{q_1,q_2}(F(x,y)), 
\end{equation}
\begin{equation}\label{eq: def N(psi)}
    \calN(X;\calR, \alpha\bmod M; \psi) = \sum_{\substack{x,y\in X\calR \\ \gcd(x,y)=1\\ (x,y)\equiv \alpha\bmod M}} \sum_{N(\fa)=F(x,y)} \psi(\fa).
\end{equation}
As in \cite{MyManinChatelet}, the contribution from the Eisenstein part (i.e. $\calM(X;\calR, \alpha \bmod M; q_1,q_2)$) will be our main term, and the cuspidal contribution will form an error term. 

\subsection{The cuspidal contribution}\label{subsec: cusp}
The proof of the upper bound for $\calN(X;\calR,\alpha\bmod M;\psi)$ follows the one presented in \cite[\S7]{MyManinChatelet} and generalized further in \cite{MyBaseChange} for cuspidal automorphic $\GL(2)$ representations that satisfy the Ramanujan-Petersson conjecture. In this proof, we place absolute values around the internal sum, i.e. 
$$\left|\sum_{N(\fa)=n}\psi(\fa)\right|,$$
and utilize the distribution of the Hecke eigenvalues to derive a small power of $\log(X)$ savings. For further discussion on why this method should be expected to produce poly-logarithmic savings, despite the absolute values, and the connection between this problem and base change for $\GL(2),$ we refer the reader to \cite{MyBaseChange}.

\begin{prop}\label{prop: cusp}
    Let $\psi\in \hat{\calC}$ and $\ord(\psi)\geq 3$. If $K_Q\not\subset K_F$, then 
    $$\calN(X;\calR,\alpha\bmod M; \psi) \ll_{F,Q,\calR} X^2 \log(X)^{-0.066}.$$
    If $K_Q\subset K_F$ and $K_F\cap \calH = K_Q$, then 
    $$\calN(X;\calR, \alpha\bmod M; \psi)\ll_{F,Q,\calR} X^2 \log(X)^{1/2}.$$
\end{prop}
\begin{proof}
    Note that 
    \begin{equation*}
        \calN(X;\calR,\alpha\bmod M; \psi) \leq \sum_{x,y\in X\calR} \left|\sum_{N(\fa)=F(x,y)} \psi(\fa)\right|.
    \end{equation*}
    Additionally, $X\calR$ will be contained in a box of sidelength $\leq c_{\calR} X$ for some constant $c_{\calR}>0$. Thus, it suffices to bound the sum 
    \begin{equation*}
        \sum_{|x|,|y|\ll X} \left|\sum_{N(\fa)=F(x,y)} \psi(\fa)\right|.
    \end{equation*}

    Since $\psi$ is a class group character of $K_Q$, it gives a $\GL_2(\A_{\Q})$ representation that we denote as $\Psi$ (specifically, $\Psi = \Ind_{W_\Q}^{W_Q} \psi$). Then  
    \begin{equation}
    \lambda_\Psi(n) = \sum_{\substack{\fa\subset \calO_F \\ N_{F/\Q}(\fa) = n}} \psi(\fa)
\end{equation}
will be the Hecke eigenvalues of $\Psi$; hence $|\lambda_\Psi(n)|$ will be a nonnegative multiplicative function. The sieve of de la Bret\`eche and Tenenbaum \cite[Theorem 1.1]{delaTenenbaumSieve} for upper bounding nonnegative multiplicative functions along squarefree binary forms gives us that 
\begin{equation*}
    \sum_{|x|,|y|\ll X} \left|\lambda_\Psi(F(x,y))\right| \ll_{F} X^2 \exp\left(\sum_{p\ll X} \frac{|\lambda_\Psi(p)|-1}{p}\cdot \#\{x\in \F_p: F(x,1) = 0\bmod p\}\right).
\end{equation*}
Here we have used that $\#\{x,y\bmod p: F(x,y)=0\bmod p\} \leq p \cdot \#\{x\bmod p: F(x,1)=0 \bmod p\}$ at all but finitely many primes $p$. Hence, we have essentially reduced our problem to the single variable case of correlation sums of $|\lambda_\Psi(n)|$ along polynomial values studied in \cite{MyBaseChange}; thus we apply the results of \cite[Theorem 1.1]{MyBaseChange} and \cite[Theorem 1.5]{MyBaseChange}; specifically that there will be a logarithmic savings if the base change of $\Psi$ to $K_F$ is cuspidal.

We want to determine when the base change $\Psi_{K_F}$ to $\GL_2(\A_{K_F})$ will remain a cuspidal representation. First, by {\cite[Theorem 2 (a),(b)]{GerardinLabesseBaseChange}}, if $K_Q\not\subset K_F$, then $\Psi_K$ will remain cuspidal. In this case, the bound from \cite[Theorem 1.1]{MyBaseChange} holds and 
\begin{equation*}
    \sum_{p\ll X} \frac{|\lambda_\Psi(p)|-1}{p}\cdot \#\{x\in \F_p: F(x,1) = 0\bmod p\} \leq -0.0665  \log\log(X) + o(\log\log X).
\end{equation*}
Hence, when $K_Q\not\subset K_F$, we get the stated bound that 
\begin{equation*}
    \sum_{|x|,|y|\ll X} |\psi_{\Psi}(F(x,y))|\ll_F X^2\log(X)^{-0.066}.
\end{equation*}

Now, if $K_Q\subset K_F$, we care about the relation of $\psi$, its kernel field, and $K_F$. Since $\psi$ is a class group character, its kernel field $L_\psi$ is contained in $\calH$, the Hilbert class field. Further, $\Gal(\calH/\calH\cap K_F) \cong \Gal(L_\psi/L_\psi\cap K_F).$ So, by \cite[Theorem 1.5]{MyBaseChange}, if $\Gal(\calH/\calH\cap K_F) \not\subset \ker(\psi^2)$ (observe that $\ker(\psi^2)$ is a proper subgroup since $\ord(\psi)\geq 3$), we still achieve a nontrivial savings that 

\begin{equation*}
    \sum_{|x|,|y|\ll X}|\lambda_\Psi(F(x,y))| \ll_F X^2 \log(X)^{1/2}.
\end{equation*}
This completes the proof of Proposition \ref{prop: cusp}.
\end{proof}

Consequently, recalling the definition of $\lambda_C(n)$ given in (\ref{eq: def cuspidal Fourier}), we know that 
\begin{equation}\label{eq: cuspidal contribution}
    \sum_{\substack{(x,y)\in X\calR \\ \gcd(x,y)=1 \\ (x,y)\equiv \alpha\bmod M}} \lambda_C(F(x,y)) \ll_{F,Q,\calR} X^2 \log(X)^{\beta_{F,Q} - 0.066}; 
\end{equation}
observe here that if $K_Q\subset K_F$ then $\beta_{F,Q}\geq 1.$ We also point out that both the conditions that $\gcd(x,y)=1$ and $(x,y)\equiv \alpha \bmod M$ are not utilized in the upper bounding procedure.

\subsection{The Eisenstein contribution}\label{subsec: eisenstein}
Finally, we wish to derive an asymptotic for (\ref{eq: def M(q_1,q_2)}) in the same vein as \cite[Proposition 8.1]{MyManinChatelet} when $t=k=1$; the main difference in our analysis is the tracking of the congruence condition $\alpha\bmod M$. However, since we allow arbitrary dependency on $\alpha\bmod M$, this does not bring any technical complications to the proof, only notational. We also point out that the argument below is \textit{not} self-contained and will cite many lemmas in \cite{MyManinChatelet}; a similar argument is also presented in \cite{delaTenenbaum-Manin} for analyzing these sums. Our goal in this subsection is to establish:

\begin{prop}\label{prop: Eisenstein}
   Let $F$ be a squarefree binary quartic form. Let $\gcd(q_1,q_2)=1$ and $q_1q_2 = \Delta(Q)$. Then for any fixed convex region $\calR$ and congruence condition $\alpha\bmod M$, 
    \begin{equation*}
        \calM(X;\calR,\alpha\bmod M;q_1,q_2) = c_{\alpha \bmod M,q_1,q_2,F,Q} X^2|\calR|\log(X)^{\beta_{F,Q}} + O_{\calR,M,F,Q}(X^2\log(X)^{\beta_{F,Q}-10^{-8}}),
    \end{equation*}
    where $\beta_{F,Q}$ is the exponent defined in Theorem \ref{thm: correlation}.
\end{prop}
\begin{proof}
First, we move the congruence condition that $(x,y)\equiv \alpha \bmod M$ into the quartic form; we do so by observing that 
\begin{multline*}
    \sum_{\substack{(x,y)\in X\calR \\ \gcd(x,y)=1\\ (x,y)\equiv \alpha\bmod M}} \varepsilon_{q_1,q_2}(F(x,y))  = \sum_{\substack{(x',y')\in M^{-1}(X\calR - \alpha)\\ \gcd(Mx'+\alpha_1,My'+\alpha_2)=1}} \varepsilon_{q_1,q_2}(F(Mx'+\alpha_1,My'+\alpha_2)).
\end{multline*}
Let us define $F_{\alpha,M} (x',y') := F(Mx'+\alpha_1,My'+\alpha_2)$; we know that $F_{\alpha,M}$ is a square binary quartic form. We also write $\calR_{\alpha,M}:= M^{-1} (X\calR - \alpha)$, which will still be a convex region. Furthermore, $|\calR_{\alpha,M}| = M^{-2}X|\calR|$ and $\partial \calR_{\alpha,M} \ll_{\calR} M^{-1} X.$ In summary, 
$$\calM(X;\calR, \alpha\bmod M; q_1,q_2) = \sum_{\substack{\bx\in \calR_{\alpha,M}\\ \gcd(Mx_1+\alpha_1, Mx_2+\alpha_2)=1}} \epsilon_{q_1,q_2}(F_{\alpha,M}(\bx)).$$

Second, we use \cite[Lemma 8.3]{MyManinChatelet} to turn the sum over all divisors $d\mid F_{\alpha,M}(\bx)$ into a sum over divisors of the irreducible factors of $F_{\alpha,M}(\bx).$ 
\begin{lemma}[{\cite[Lemma 8.3]{MyManinChatelet}}]\label{lem: multiplicativity of eisenstein}
Let $n$ and $m$ be integers satisfying $\gcd(nm,\Delta)=1$.
Then $$\varepsilon_{q_1,q_2}(nm) = \sum_{\substack{c\mid \gcd(m,n)}} \mu(c)\chi(c) \varepsilon_{q_1,q_2}(n/c) \varepsilon_{q_1,q_2}(m/c).$$    
\end{lemma}
By the reduction of \S8.1 of \cite{MyManinChatelet}, we have that 
\begin{equation*}
\calM(X;\calR,\alpha\bmod M; q_1,q_2) = \sum_{\bc} \mu(\bc)\chi(\bc) S(X;\calR,\alpha\bmod M; \bc),   
\end{equation*}
where the sum is over tuples $\bc$ with $c_{ij}\mid \text{Res}(F_i,F_j)$ for $F_1,...,F_r$ irreducible factors of $F_{\alpha,M}$ and 
\begin{equation}\label{eq: def of S(c)}
    S(X;\calR,\alpha\bmod M; \bc) = \sum_{\substack{\bx\in \calR_{\alpha,M}\\ \gcd(Mx_1+\alpha_1,M x_2+\alpha_2)=1\\ c_i\mid F_i(\bx)}} \prod_{i=1}^r \sum_{d_i\mid F_i(\bx)} \chi_{q_1}(d_i)\chi_{q_2}(F_i(\bx)/d_ic_i).
\end{equation}

Third, we follow \S8.2 of \cite{MyManinChatelet} and relate $\chi_{q_1}\star \chi_{q_2}$ with the Dirichlet convolution $1\star\chi$, where $\chi$ is the quadratic character generating $K_Q$. We obtain that 
\begin{multline*}
    S(X;\calR,\alpha\bmod M;\bc) \\ = \sum_{\substack{n\ll X\\ p\mid n\implies p\mid \Delta}}\chi_{q_1}(p_{q_2}(n))\chi_{q_2}(p_{q_1}(n)) \mathbf{1}_{\exists \fa:N(\fa)=n}\sum_{\substack{\ba\bmod \Delta\\ \gcd(a_i,\Delta)=1}} \prod_{i=1}^r \chi_{q_1}(a_i) S(X;\calR, \alpha\bmod M; \bc,n,\ba),
\end{multline*}
where we define 
\begin{equation}\label{eq: def of S(c,n,a)}
    S(X;\calR, \alpha\bmod M; \bc,n,\ba) = \sum_{\substack{\bx\in \calR_{\alpha,M}\\ \gcd(Mx_1+\alpha_1,M x_2+\alpha_2)=1\\ c_i\mid F_i(\bx) \\ p_{\Delta}(F_{\alpha,M}(\bx))=n \\ p_{\neg \Delta}(F_{i}(\bx)/c_i) \equiv a_i\bmod n\Delta}} \prod_{i=1}^r (1\star \chi)(p_{\neg \Delta}(F_i(\bx)/c_i)). 
\end{equation}

Next, we apply M\"obius inversion to remove the gcd condition: 
\begin{equation*}
    S(X;\calR,\alpha\bmod M;\bc,n,\ba) = \sum_{b\ll_{\calR} X} \mu(b) S(X;\calR,\alpha\bmod M;\bc,n,\ba,b),
\end{equation*}
where we define
\begin{equation}\label{eq: def of S(c,n,a,b)}
    S(X;\calR,\alpha\bmod M;\bc,n,\ba,b) = \sum_{\substack{\bx\in \calR_{\alpha,M}\\ b\mid M\bx+\alpha \\ c_i\mid F_i(\bx) \\ p_{\Delta}(F_{\alpha,M}(\bx))=n \\ p_{\neg \Delta}(F_{i}(\bx)/c_i) \equiv a_i\bmod n\Delta}} \prod_{i=1}^r (1\star \chi)(p_{\neg \Delta}(F_i(\bx)/c_i))
\end{equation}
By \cite[Lemma 8.4, Lemma 8.6]{MyManinChatelet} with $t=k=1$, $L=1$, we can bound away the contribution from those $n,b\geq \log(X)^{-10^{-7}}$ to a suitable error term. \\

We now recall that in \cite{MyManinChatelet} to evaluate $S(X;\calR,\alpha\bmod M; \bc, n,\ba,b)$, we needed to establish two results: a level of distribution bound and an upper bound on the contribution of the large moduli. Define for a fixed set of divisors $\bd = (d_i)_{i=1}^r$:
\begin{equation}\label{eq: def of S(d;c,n,a,b}
T(\bd; X;\calR, \alpha\bmod M; \bc,n,\ba,b) = \#\left\{\bx\in \calR_{\alpha,M}: \begin{array}{c}
     b\mid M\bx+\alpha, \\ c_id_i\mid F_i(\bx),  \\
     p_{\Delta}(F_{\alpha,M}(\bx))=n, \\ p_{\neg \Delta} (F_i(\bx)/c_i) \equiv a_i \bmod n\Delta
\end{array}\right\}.
\end{equation}

First, we would like a nearly-sufficient level of distribution bound for our binary form $F_{\alpha,M}$. Since we allow arbitrary dependence on our congruence condition $\alpha\bmod M$, we can establish this with careful lattice estimates. 
\begin{lemma}\label{lem: LoD}
    Let $F_{\alpha,M} = \prod_{i=1}^r F_i$ be a squarefree binary form of degree $\leq 4$, where $F_i(x,y)$ are irreducible factors of $F(x,y)$. For a tuple $\mathbf{e} = (e_i)$, we denote the product of the entries as $e = e_1...e_r$. For a fixed $n,\ba,\bc, \mathbf{d}=(d_i)$ and $b$, we define the following local function: 
    \begin{multline*}\varrho_{F}(\mathbf{d};\alpha\bmod M;\bc,n,\ba,b) := \#\{\bx\bmod bcdn\Delta: c_id_i\mid F_i(x,y),    b\mid \gcd(Mx_1+\alpha_1,Mx_2+\alpha_2), \\ p_{\Delta}(F_{\alpha,M}(x,y)) = n, p_{\neg \Delta}(F_i(x,y)/c_i)\equiv  a_i\bmod \Delta\}.\end{multline*}
    For a tuple $\bD = (D_i)_{i=1}^r$, we define the error quantity as: 
    $$E(\bD) = \sum_{d_i\sim D_i} \left|T(\bd; X;\calR,\alpha\bmod M; \bc,n,\ba,b) - |\calR_{\alpha,M}|\cdot \frac{\varrho_{F}(\mathbf{d};\alpha\bmod M;\bc,n,\ba,b)}{(bcdn\Delta)^2}\right|.$$
    If $F(x,y)$ has no linear factors, then we have that for an explicit constant $\gamma_r>0$, 
    \begin{multline*}
        E(\bD) \ll_{\calR,M,F,Q}  D \cdot \log\left(1+\frac{X^2}{D}\right)^9 +
        \frac{D}{b^2}\cdot \exp(\gamma_r \sqrt{\log_2(X)\log_3(X)})\\
        + \frac{\tau(b^4)}{b} \cdot X\sqrt{D} \cdot \exp(\gamma_r\sqrt{\log_2(X)\log_3(X)}).
    \end{multline*}
    If $F(x,y)$ has a linear factor, then we have that 
    \begin{multline*}E(\bD) \ll_{\calR,M,F,Q} D \cdot \log\left(1+\frac{X^2}{D}\right)^9 +
         \frac{D}{b^2}\cdot \exp(\gamma_r \sqrt{\log_2(X)\log_3(X)})\\
        +  \frac{\tau(b^4)}{b} \cdot X\sqrt{D} \cdot \exp(\gamma_r\sqrt{\log_2(D)\log_3(D)}) \\ 
        + \frac{X}{b}\cdot \max_{i:\deg(F_i)=1} D_i\log(D)^2 \log\left(1+\frac{X^2}{D}\right)^8.
    \end{multline*}
\end{lemma}
\begin{proof}
    Since $\calR_{\alpha,M}$ is a convex region satisfying that $\partial \calR_{\alpha,M} \ll_{\calR,M} X$, we apply the lattice point counting argument of \cite[Proposition 9.1]{MyManinChatelet}, taking $L=t=k=1$ and replacing the box $B^{1/2}\mathcal{B}(\bx_0,L)$ with our region $\calR_{\alpha,M}$. Finally, we recall that $c = O_F(1)$ and remove the dependency on $c$ in the final bounds.  
\end{proof}

In the above level of distribution, since $|\calR_{\alpha,M}| = (X/M)^2|\calR|,$ $E(\bD)$ is sufficiently small for $D\leq X^2 \log(X)^{-10^{-5}}$ if $F(x,y)$ has no linear factors, and for $D_i\leq X^{1/2}\log(X)^{-1-10^{-5}}$ if $F(x,y)$ has a linear factor. So it remains to bound the contribution from the large moduli. For this, we cite two bounds from \cite[\S10]{MyManinChatelet} with $L=t=k=1$ and $B^{1/2}=X$. 

\begin{lemma}[{\cite[Proposition 10.1]{MyManinChatelet}}]\label{lem: large moduli 1}
    Let $F_{\alpha,M}$ be a squarefree binary form. Let $\bD = (D_i)_{i=1}^r$ satisfy that $\log(D)\gg \log(X)$.
Then the following bound holds:
\begin{multline*}
    \sum_{d_i\sim D_i} T(\bd;X;\calR,\alpha\bmod M;\bc,n,\ba,b) \\ \ll_{\calR,M,F,Q}  \left(\frac{X^2}{b^2} + \frac{X^{3/2}}{b}\right) \cdot  \log(X)^{\beta_{F,Q}-10^{-5}}\cdot  \left(\frac{X^2}{D} + \frac{D}{X^2}+\log\log(X)\right)^{1/2} . 
\end{multline*}
Here $\beta_{F,Q}$ is defined as in Theorem \ref{thm: correlation}.
\end{lemma}
\begin{proof}
    We make the remark that again replacing the box $B^{1/2}\mathcal{B}(\bx_0,L)$ with the region $\calR_{\alpha,M}$ brings no extra complications since $\calR_{\alpha,M}$ can be contained in a box of side length $O_{\calR,M}(X)$. We also note that in the notation of \cite[Proposition 10.1]{MyManinChatelet}, $\beta_{F,Q} = \varrho_{\Delta,F}-2$ and that $\varrho_{\Delta,F}/2-1 \leq \varrho_{\Delta,F}-2.$
\end{proof}

\begin{lemma}{{\cite[Proposition 10.2]{MyManinChatelet}}}\label{lem: large moduli 2}
    Let $F_{\alpha,M}$ be a squarefree binary form with a linear factor, denoted as $F_1(x,y)$. Then, the following bound holds:
    $$\sum_{d_1\sim D_1}\sum_{\substack{d_i\ll X^{\deg(F_i)/2}\\ i=2,...,r}} T(\bd;X;\calR,\alpha\bmod M; \bc,n,\ba,b) \ll_{\calR,M,F,Q} X^2 \cdot \log(X)^{\beta_{F,Q}-10^{-5}}.$$
\end{lemma}
\begin{proof}
    Since $\calR_{\alpha,M}$ can be contained in a box of side length $O_{\calR,M}(X)$, the proof can proceed as in \cite[Proposition 10.2]{MyManinChatelet}. 
\end{proof} 

Combining Lemmas \ref{lem: LoD}, \ref{lem: large moduli 1}, and \ref{lem: large moduli 2} as in \cite[\S11]{MyManinChatelet}, we achieve that 
\begin{multline*}
    S(X;\calR,\alpha\bmod M;\bc,n,\ba,b) = |\calR_{\alpha,M}| \sum_{d_i\ll X^{\deg(F_i)/2}} \prod_{i=1}^r \chi(d_i) \cdot \frac{\varrho_F(\bd;\alpha\bmod M;\bc,n,\ba,b)}{(bcdn\Delta)^2}\\
    + O_{\calR,M,F,Q} \left(X^2 \log(X)^{\beta_{F,Q}-10^{-6}}\right).
\end{multline*}
Since we can separately bound the contribution from the large $b\geq \log(X)^{10^{-7}}$, we can sum over all $b\ll_\calR X$. Hence, 
\begin{multline*}
    \sum_{b\ll_\calR X} \mu(b) S(X;\calR,\alpha\bmod M;\bc,n,\ba,b) \\ = |\calR_{\alpha,M}| \sum_{b} \frac{\mu(b)}{b^2}\sum_{d_i\ll X^{\deg(F_i)/2}} \prod_{i=1}^r \chi(d_i) \cdot \frac{\varrho_F(\bd;\alpha\bmod M;\bc,n,\ba,b)}{(cdn\Delta)^2} + O_{\calR,M,F,Q}(X^2\log(X)^{\beta_{F,Q}-10^{-7}}).
\end{multline*}
By the computations in \cite[\S11]{MyManinChatelet}, we can rearrange this main term to be 
\begin{equation*}
    \sum_{d_i\ll X^{\deg(F_i)/2}} \prod_{i=1}^r \chi(d_i) \cdot \frac{\varrho_F^*(\bd;\alpha\bmod M;\bc,n,\ba)}{(cdn\Delta)^2},
\end{equation*}
where 
\begin{equation*}
    \varrho_F^*(\bd;\alpha\bmod M;\bc,n,\ba) = \#\left\{x,y\bmod cdn\Delta: \begin{array}{c}
         \gcd(Mx+\alpha_1, My+\alpha_2, cd)=1,\\
         p_{\Delta}(F_{\alpha,M}(x,y)) =n, \\ p_{\neg\Delta}(F_i(x,y)/c_i)\equiv a_i\bmod \Delta
    \end{array}\right\}.
\end{equation*}

Finally, it remains to consider the sum over $d_i$. Consider for a fixed $\bc,n,\ba$: 
\begin{equation*}
   \xi(s;\alpha\bmod M; \bc,n,\ba) = \sum_{d_i}\prod_{i=1}^r\chi(d_i) \cdot\frac{\varrho_F^*(\bd;\alpha\bmod M; \bc,n,\ba,b)}{(cn\Delta)^2 d^{2s}}.
\end{equation*}
It is straightforward to see that since for primes $p_i$, the expression $\varrho_F((p_i^{e_i}); \alpha\bmod M; \bc,n,\ba)\ll_r p_1^{e_1}\dots p_r^{e_r}$, we have that the Dirichlet series $\xi(s;\alpha\bmod M,\bc,n,\ba)$ converges absolutely for $\Re(s)>1$. Furthermore it has an Euler product expansion and that
\begin{equation}
\xi(s;\alpha\bmod M; \bc,n,\ba) = R_{\alpha \bmod M,\bc,n,\ba}(s) \prod_{i=1}^r \zeta_{K_{F_i}}(s,\chi),
\end{equation}
where $R_{\alpha \bmod M,\bc,n,\ba}(s)$ is a factor that converges for $\Re(s)>3/4$, the field $K_{F_i}$ is the splitting field of $F_i$, and $\zeta_{K_{F_i}}(s,\chi)$ is the twist of Dedekind zeta function of $K_{F_i}$ by the quadratic character $\chi$ mod $\Delta.$ This decomposition is discussed in more detail in \cite[\S11.5]{MyManinChatelet} in the evaluation of $V(\mathbf{k},\bc,D)$.

Thus, we can rewrite the expression as 
\begin{multline*}
    S(X;\calR,\alpha\bmod M;\bc,n,\ba) = |\calR_{\alpha,M}| \frac{R_{\alpha \bmod M,\bc,n,\ba}(1)}{(cn\Delta)^2} \cdot \prod_{i=1}^r \left(\textrm{Res}_{s=1} \zeta_{K_{F_i}}(s,\chi) \cdot \log(X)^{\mathbf{1}_{K_Q\subset K_{F_i}}}\right) \\ + O_{\calR,M,F,Q}\left(X^2\log(X)^{\beta_{F,Q}-10^{-7}}\right). 
\end{multline*}
Here we have used $\textrm{Res}_{s=1} \zeta_{K_{F_i}}(s,\chi)$ to denote the residue at $s=1$ of the $L$-function, or if $\zeta_{K_{F_i}}(s,\chi)$ is analytic at $s=1$ then its value. Recall at this point that $|\calR_{\alpha,M}| = (X/M)^2 |\calR|.$

Finally, we observe the relation that 
\begin{equation}
    \beta_{F,Q} = \sum_{i=1}^r \mathbf{1}_{K_Q\subset K_{F_i}}.
\end{equation}
In the expression for $S(X;\calR,\alpha\bmod M; \bc,n,\ba)$ above, the exponent of $\log(X)$ is $\beta_{F,Q}$ and additionally independent of $(\bc,n,\ba).$ Hence, we have that 
\begin{multline*}
    \calM(X;\calR,\alpha\bmod M;q_1,q_2) = |\calR_{\alpha,M}|\log(X)^{\beta_{F,Q}} \cdot \prod_{i=1}^r \text{Res}_{s=1}\zeta_{K_{F_i}}(s,\chi) \\ \times \sum_{\bc}\frac{\mu(\bc)\chi(\bc)}{c^2} \sum_{\substack{n\leq \log(X)^{10^{-7}}\\ p\mid n\implies p\mid \Delta}} \frac{\chi_{q_1}(p_{q_2}(n)) \chi_{q_2}(p_{q_1}(n))}{n^2} \sum_{\substack{\ba\bmod \Delta\\ \gcd(a_i,\Delta)=1}} \frac{\prod_{i=1}^r \chi_{q_1}(a_i)}{\Delta^2} R_{\alpha \bmod M,\bc,n,\ba}(1)\\
    +\sum_{\substack{n\leq \log(X)^{10^{-7}} \\ p\mid n \implies p\mid \Delta}} O_{\calR,M,F,Q}(X^2\log(X)^{\beta_{F,Q}-10^{-7}}).
\end{multline*}
Observe that we have used above that the sums over $\bc$ and $\ba$ are finite and the number of terms only depends on $F$ and $Q$. 

To establish the asymptotic of Proposition \ref{prop: Eisenstein}, it remains to check that \begin{equation}\label{eq: sum over n bound}\sum_{\substack{n\leq \log(X)^{10^{-7}}\\ p\mid n\implies p\mid \Delta}} X^2\log(X)^{\beta_{F,Q}-10^{-7}} \ll X^2\log(X)^{\beta_{F,Q}-10^{-8}},\end{equation}
and that for any fixed $\bc$ and $\ba$,
\begin{equation}\label{eq: sum over n to converge}\sum_{\substack{n\in \mathbb{N}\\ p\mid n\implies p\mid \Delta}} \chi_{q_1}(p_{q_2}(n))\chi_{q_2}(p_{q_2}(n))\cdot \frac{R_{\alpha \bmod M,\bc,n,\ba}(1)}{n^2}\end{equation}
converges absolutely to a finite constant. 
For both of the estimates above, we recall that the sum over $n$ is over integers satisfying that $p\mid n\implies p\mid \Delta$. Consequently these are sums over $O_Q(\log\log(X)^{\omega(Q)})$ terms and this establishes the upper bound \eqref{eq: sum over n bound} above.

Since $\gcd(d,\Delta)=1$ and $p\mid n\implies p\mid \Delta$, we have that $\gcd(d,n)=1$. It follows from the Chinese remainder theorem for the solutions counted by $\varrho_F(\bd;\alpha\bmod M; \bc,n,\ba)$ that we can bound 
$$R_{\alpha \bmod M,\bc,n\ba}(1)\ll_{F,Q,M} \varrho_F(1;\alpha\bmod M; \bc,n,\ba).$$
After loosening the condition that $p_{\Delta}(F_{\alpha,M}(x,y))=n$ to requiring that $n\mid F_{\alpha,M}(x,y)$, it is clear that for $n\gg_{F,Q,M} 1$, we have that $\varrho_F(1;\alpha\bmod M; \bc,n,\ba) \ll_{F,Q,M} n.$ Thus, we have that 
\begin{equation*}
    \sum_{\substack{n\in \mathbb{N}\\ p\mid n\implies p\mid \Delta}} \left|\chi_{q_1}(p_{q_2}(n))\chi_{q_2}(p_{q_2}(n))\frac{R_{\alpha \bmod M,\bc,n,\ba}(1)}{n^2}\right|\ll_{M,\bc,\ba,F,Q} \sum_{\substack{n\in \mathbb{N}\\ p\mid n\implies p\mid \Delta}} \frac{1}{n}. 
\end{equation*}
The sum over those $n$ such that $p\mid n \implies p\mid \Delta$ is a product of convergent geometric series, and hence converges. This gives us that \eqref{eq: sum over n to converge} will indeed converge absolutely to a finite constant (depending on $\ba,\bc,F,Q$ and $\alpha\bmod M$). Since we allow arbitrary dependence over $\ba,\bc,F,Q$ and $\alpha\bmod M$,
this completes the proof of Proposition \ref{prop: Eisenstein}

\end{proof}
\begin{remark}
    Since we do not attempt to evaluate this constant here, we only require \eqref{eq: sum over n to converge} to converge absolutely. In \cite[\S12]{MyManinChatelet}, the equivalent expression ($\nabla_{q_1,q_2}(\bc)$) is analyzed further and shown to be nonzero if and only if there are no local obstructions to rational points on an auxiliary variety $Y^*_{\Delta,\bc,f_1,\dots,f_r}$; however, in this situation $R_{\alpha \bmod M,\bc,n,\ba}(1)$ is more complicated to analyze due to the extra congruence condition. 
\end{remark}

\subsection{Proof of Theorem \ref{thm: correlation}}
Recalling that 
$$r_Q(n) = \frac{|\calU|}{|\calC|} \sum_{\psi\in \hat{\calC}} \overline{\psi([Q])} \sum_{N(\fa)=n} \psi(\fa),$$
and applying Proposition \ref{prop: cusp} (for $\ord(\psi)\geq 3)$ and Proposition \ref{prop: Eisenstein} (for $\ord(\psi)\leq 2)$, we establish that
\begin{align*}
    \sum_{\substack{x,y\in X\calR \\ \gcd(x,y)=1 \\ (x,y)\equiv \alpha\bmod M}}r_Q(F(x,y)) &= 
    \frac{|\calU|}{|\calC|} \sum_{\substack{q_1q_2=\Delta \\ \gcd(q_1,q_2)=1}}' c_{\alpha\bmod M, q_1,q_2,F,Q} X^2 |\calR| \log(X)^{\beta_{F,Q}} \\ &+ O_{\calR,M,F,Q}(X^2\log(X)^{\beta_{F,Q}-10^{-8}})  + 
    \frac{|\calU|}{|\calC|} \sum_{\substack{\psi\in \hat{C}\\ \ord(\psi)\geq 3}}  O_{F,Q,\calR}(X^2\log(X)^{\beta_{F,Q}-0.066}) 
    \\  
    &= c_{\alpha, F,Q} X^2|\calR| \log(X)^{\beta_{F,Q}} + O_{\calR, F,Q,M} (X^2\log(X)^{\beta_{F,Q}-10^{-8}}),
\end{align*}
where we take the constant $c_{\alpha,F,Q}$ to be $$c_{\alpha,F,Q} = \frac{|\calU|}{|\calC|} \sum_{\substack{q_1q_2=\Delta \\ \gcd(q_1,q_2)=1}}' c_{\alpha\bmod M, q_1,q_2,F,Q},$$
where this final sum is a (finite) restricted sum over genus characters for $K_Q$.\qed

\section{Computations of the Picard group}

In this section, we will compute the Picard group for general values of $A$, $B$, and $Q(u,v)$ to establish the expected growth rate for Manin's conjecture (Conjecture \ref{conj: Manin}).

\subsection{Computations of the Picard group}\label{sec: Picard}

We want to compute the rank $\varrho$ of the Picard group over $\Q$ of the desingularization of $S_{A,B,Q}$ and hence the exponent of $\log(T)$ in Conjecture \ref{conj: Manin}. First, we consider the geometry of the surface $S_{A,B,Q}$ over $\overline{\Q}$. Let $\alpha,\overline{\alpha}$ be the two roots of the quadratic form $Q(x,1)=0$. We can see that there two singularities given by the points $[0,0,\alpha,1]$ and $[0,0,\overline{\alpha},1]$ sitting in the weighted projective space $\bP(2,3,1,1)$. Additionally, we have that there are three lines given by $$L_i = \{y=0, x-\beta_iQ(u,v)=0\}, i=1,2,3,$$
where $\beta_i$ are roots of the cubic $x^3 +Ax +B = 0$, and the two lines in $\bP(2,3,1,1)$: $$E_{+} = \{u-\alpha v=0, y^2 = x^3\}, E_{-} = \{u-\overline{\alpha}v=0, y^2 = x^3\}.$$
We exclude $L_1,L_2,L_3,E_1,E_2$ by restricting to the open set $U_1\cap U_2$ as defined in Theorem \ref{thm: Manin}. 

Next, we will compute the minimal desingularization of $S_{A,B,Q}$ over $\overline{\Q}$. Then, we will find the rank of $$\Pic_{\Q}(\Tilde{S}) = \Pic_{\Q(\alpha)}(\Tilde{S})^{\Gal(\Q(\alpha)/\Q)}.$$

\subsubsection{Computations over $\overline{\Q}$}
We denote the minimal desingularization of $S_{A,B,Q}$ as $\Tilde{S}_{A,B,Q}$. To compute the desingularization, we first will blow up the singular point $[0,0,\alpha,1]$; we denote this blow up as $S_{\alpha}$. We look at the image in three charts: $D(x)$, $D(y)$ and $D(u-\alpha)$. In $D(x)$, the defining equation for the blowup is given by $$y_1^2 - x(1+A(u_1x+\alpha-\overline{\alpha})^2 u_1^2 + B(u_1x+\alpha-\overline{\alpha})^3u_1^3))=0,$$ where we have written $y=y_1x$ and $u-\alpha = u_1 x$. So, the exceptional divisor is given by $$\{y_1=0, x=0\}$$ where $u_1$ is allowed to range over $\overline{\Q}$. 
We recall that $\alpha-\overline{\alpha}= \sqrt{\Delta(Q)}.$ Let $z_i$ denote the roots of the equation $1+A\Delta(Q) z^2 + B(\sqrt{\Delta(Q)})^3 z^3=0$. Then there are double points on the exceptional divisor at $y_1=0, x=0$ and $u_1 = z_i$ for $i=1,2,3$. 

 Next, we look at the chart $D(y)$; here, the defining equation is given by $$y^2 ( 1+ y(x_1^3 + A(u_1x+\alpha-\overline{\alpha})^2 (x_1u_1^2) + B(u_1x+\alpha-\overline{\alpha})^3 (u_1)^3)= 0,$$ where we have taken $x=x_1y$ and $u-\alpha = u_1y$. We can see that there are no additional singularities in this chart. 
 
 In our final chart, $D(u-\alpha)$, we get the defining equation $$(u-\alpha)^2(y_1^2 - (u-\alpha)(x_1^3 + A(u-\overline{\alpha})^2 x_1 + B(u-\overline{\alpha})^3)=0.$$
    The exceptional divisor is given by $$\{y_1=0, u=\alpha, x_1\in \overline{\Q}\}.$$ There are double points given by solutions $\lambda_i$ to $\lambda^3 + A\Delta(Q) \lambda+ B(\sqrt{\Delta(Q)})^3=0$. 
    
    Thus, we can see that in our blow up $S_{\alpha}$, there is an exception divisor $\bP^1$, which we will denote as $\Gamma_{\alpha}$. Additionally, there are three singularities corresponding to the roots $z_i$ of $z^3 + A\Delta(Q) z + B(\sqrt{\Delta(Q)})^3.$ We note that since we have assumed that $4A^3-27B^2\neq 0$, these are three distinct roots.
    
So, we blow up $S_{\alpha}$ at each of these singularities corresponding to $z_1,z_2,$ and $z_3$ to form $S_{\alpha}'$. We can see that this blowup will have three additional exceptional divisors $\Gamma_{\alpha,1},\Gamma_{\alpha,2}, \Gamma_{\alpha,3}$ isomorphic to $\bP^1$ that do not intersect each other and intersect $\Gamma_{\alpha}$ each at one point. Additionally, all four of these divisors will be $(-2)$-curves, i.e. $\Gamma_i\cdot \Gamma_i = -2$, since they are blow ups of double points. 

Next, we can blow up $S_{\alpha}'$ at $[0,0,\overline{\alpha},1]$ once to create $S_{\overline{\alpha}}'$. Analogously to the above computation, we will have an exceptional divisor $\Gamma_{\overline{\alpha}}$ which is isomorphic to $\bP^1$, and three double points corresponding to roots of $z^3 + A\Delta(Q) z + B(-\sqrt{\Delta(Q)})^3=0$. We then have a final blow up at each of these singularities to create $S_{\overline{\alpha}}'' = \Tilde{S}_{A,B,Q}$. This process generates four new exceptional divisors, which we label $\Gamma_{\overline{\alpha}}, \Gamma_{\overline{\alpha},1},\Gamma_{\overline{\alpha},2},\Gamma_{\overline{\alpha},3}$. $\Gamma_{\overline{\alpha},1},\Gamma_{\overline{\alpha},2}, \Gamma_{\overline{\alpha},3}$ will be isomorphic to $\bP^1$, will not intersect each other, and will each intersect $\Gamma_{\overline{\alpha}}$ at one point. Again, these divisors will be $(-2)$-curves. 

We can summarize the above findings by stating that $S_{A,B,Q}$ has two $D_4$ singularities. In particular, we get the following extended Dynkin diagram where the two $D_4$ singularities are pictured in black. The lines $L_i$ pictured in dashed red will intersect both $\Gamma_{\alpha,i}$ and $\Gamma_{\overline{\alpha},i}$, and the lines $E_+, E_-$ pictured in dash blue will intersect $\Gamma_{\alpha}$ and $\Gamma_{\overline{\alpha}}$ respectively. 

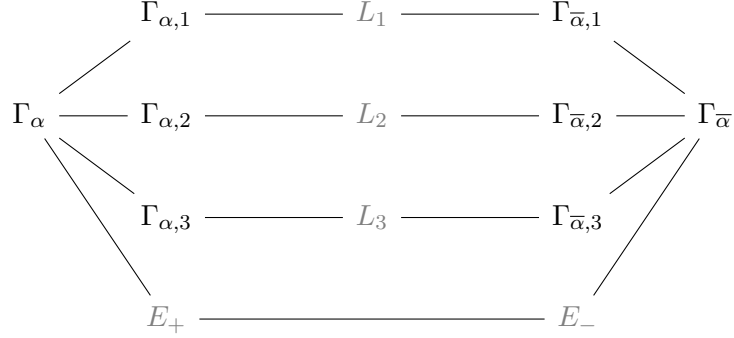
\[\begin{tikzcd}
	& {\Gamma_{\alpha,1}} && {\textcolor{gray}{L_1}} && {\Gamma_{\overline{\alpha},1}} \\
	{\Gamma_\alpha} & {\Gamma_{\alpha,2}} && {\textcolor{gray}{L_2}} && {\Gamma_{\overline{\alpha},2}} & {\Gamma_{\overline{\alpha}}} \\
	& {\Gamma_{\alpha,3}} && {\textcolor{gray}{L_3}} && {\Gamma_{\overline{\alpha},3}} \\
	& {\textcolor{gray}{E_+}} &&&& {\textcolor{gray}{E_-}}
	\arrow[no head, from=2-2, to=2-4]
	\arrow[no head, from=1-2, to=1-4]
	\arrow[no head, from=1-4, to=1-6]
	\arrow[no head, from=2-4, to=2-6]
	\arrow[no head, from=3-2, to=3-4]
	\arrow[no head, from=3-4, to=3-6]
	\arrow[no head, from=2-1, to=1-2]
	\arrow[no head, from=2-1, to=2-2]
	\arrow[no head, from=2-1, to=3-2]
	\arrow[no head, from=1-6, to=2-7]
	\arrow[no head, from=2-6, to=2-7]
	\arrow[no head, from=3-6, to=2-7]
	\arrow[no head, from=2-1, to=4-2]
	\arrow[no head, from=4-2, to=4-6]
	\arrow[no head, from=4-6, to=2-7]
\end{tikzcd}\]

Since $S_{A,B,Q}$ has two $D_4$ singularities, we have that $\Tilde{S}_{A,B,Q}$, as computed above, is the minimal desingularization of $S$. 
Since $\Tilde{S}_{A,B,Q}$ is the minimal desingularization of a del Pezzo surface of degree one, we know that $$ \rk\Pic_{\overline{\Q}}(\Tilde{S}_{A,B,Q}) = 10-1 = 9$$
(see, for example, \cite{Derenthal}).
Let us denote by $H$ the hyperplane section of $\Tilde{S}_{A,B,Q}.$ We claim that the exceptional divisors computed above, along with $H$, generate the entire Picard group over $\overline{\Q}.$

\begin{lemma}
$$\mathrm{Pic}_{\overline{\Q}}(\Tilde{S}) = \langle H, \Gamma_{\alpha},\Gamma_{\alpha,1},\Gamma_{\alpha,2},\Gamma_{\alpha,3}, \Gamma_{\overline{\alpha}}, \Gamma_{\overline{\alpha},1}, \Gamma_{\overline{\alpha},2},\Gamma_{\overline{\alpha},3}\rangle.$$
\end{lemma}
\begin{proof}
The hyperplane section will be indepedent of any of the exceptional divisors, so we will look at the intersection numbers of the eight exceptional divisors. All of the six divisors coming from double points will have self-intersection $-2$, so: 
$$\Gamma_{\beta,i}\cdot \Gamma_{\beta,i} = -2,$$ where we are using $\beta$ to denote either $\alpha$ or $\overline{\alpha}.$ 
Additionally, we have that $$\Gamma_{\beta}\cdot \Gamma_{\beta,i}=1$$ and $$\Gamma_{\beta}\cdot \Gamma_{\beta} = -2.$$

Next, we note that we will have $$\Gamma_{\alpha,i}\cdot \Gamma_{\overline{\alpha},j} = 0$$ for any $i,j\in \{\bot,1,2,3\},$ since these are two separate singularities.
Thus, we get the following intersection matrix: 

\begin{equation*}
    \begin{blockarray}{ccccccccc}
    \Gamma_{\alpha} & \Gamma_{\alpha,1} & \Gamma_{\alpha,2} & \Gamma_{\alpha,3} & \Gamma_{\overline{\alpha}} & \Gamma_{\overline{\alpha},1} & \Gamma_{\overline{\alpha},2} &\Gamma_{\overline{\alpha},3} \\
    \begin{block}{(cccccccc)c}
      -2 & 1 & 1 & 1 & 0 & 0 & 0 & 0 & \Gamma_{\alpha} \\
      1 & -2 & 0 & 0 & 0 & 0 & 0 & 0 & \Gamma_{\alpha,1} \\
      1 & 0 & -2 & 0 & 0 & 0 & 0 & 0 & \Gamma_{\alpha,2} \\
      1 & 0 & 0 & -2 & 0 & 0 & 0 & 0 & \Gamma_{\alpha,3} \\
      0 & 0 & 0 & 0 & -2 & 1 & 1 & 1 & \Gamma_{\overline{\alpha}} \\
      0 & 0 & 0 & 0 & 1 & -2 & 0 & 0 & \Gamma_{\overline{\alpha},1} \\
      0 & 0 & 0 & 0 & 1 & 0 & -2 & 0 & \Gamma_{\overline{\alpha},2} \\
      0 & 0 & 0 & 0 & 1 & 0 & 0 & -2 & \Gamma_{\overline{\alpha},3} \\
    \end{block}
    \end{blockarray}
\end{equation*}
This matrix has a nonzero determinant, and so the exceptional divisors are linearly independent. Thus, along with $H$, we have that these exceptional divisors must generate $\Pic_{\overline{\Q}}(\Tilde{S}_{A,B,Q}).$

\end{proof}

Since we understand $\Pic_{\overline{\Q}}(\Tilde{S})$, we aim to understand the Galois action on it. In particular, we want to compute $$\Pic_{\Q}(\Tilde{S}) = \Pic_{\overline{Q}}(\Tilde{S})^{\Gal(\overline{\Q}/\Q)} = \Pic_{\Q(\alpha)}(\Tilde{S})^{\Gal(\Q(\alpha)/\Q)}.$$
For considering the orbits of $\Gal(\overline{\Q}/\Q)$, we note that it suffices to look at the action of $\Gal(K/\Q)$, where $K$ is a splitting field over $\Q(\alpha)$ of the polynomial $z^3 + A\Delta(Q) z + B(\sqrt{\Delta(Q)})^3.$ It is clear that $[H]$ and $[\Gamma_{\alpha}+\Gamma_{\overline{\alpha}}]$ will be invariant under this action. Additionally, we have that $[\Gamma_{\alpha,1}]+[\Gamma_{\alpha,2}]+[\Gamma_{\alpha,3}] + [\Gamma_{\overline{\alpha},1}] + [\Gamma_{\overline{\alpha},2}]+[\Gamma_{\overline{\alpha},3}]$ will invariant under the Galois action. Hence, we can see that $$\rk \Pic_{\Q}(\Tilde{S}) \geq 3.$$

To determine the exact rank, we find that it is clearer to look at the picture over $\Q(\alpha)$.

\subsubsection{Computation over $\Q(\alpha)$}
Over $\Q(\alpha)$, we note that nothing changes in the first blowup to $S_\alpha$. Then for the second blowup, rather than getting an exceptional divisor for each root of $P(z) = z^3 + A\Delta(Q) z + B(\sqrt{\Delta(Q)})^3$, we have a divisor corresponding to each irreducible factor of the polynomial over $\Q(\alpha)$ (as these will correspond to prime ideals). 
In particular, if $P(z)$ is irreducible, then we have the divisor $\Gamma_{\alpha, P(z)}$. If $P(z) = Q(z) L(z)$, where $Q(z)$ is an irreducible quadratic, then we have $\Gamma_{\alpha, Q(z)}$ and $\Gamma_{\alpha, L(z)}$. Finally, if $P(z) = (z - r_1) (z-r_2)(z-r_3)$, then we have a divisor corresponding to each root, $\Gamma_{\alpha, (z-r_1)}, \Gamma_{\alpha,(z-r_2)},\Gamma_{\alpha,(z-r_3)}.$
For the blowup over $[0,0,\overline{\alpha},1]$, the same picture occurs except now we consider the factorization of $\overline{P(z)}$. 

\subsubsection{Computation over $\Q$} Finally, we want to consider the action of $\Gal(\Q(\alpha)/\Q)$ on the divisors over $\Q(\alpha)$. We can see that $$\Gamma_{\alpha}\mapsto \Gamma_{\overline{\alpha}}$$ $$\Gamma_{\alpha, P(z)} \mapsto
\Gamma_{\overline{\alpha},\overline{P(z)}}.$$
So, if $P(z)$ is irreducible in $\Q(\alpha)$, we have that $$\Pic_\Q(\Tilde{S}) = \langle H, \Gamma_\alpha + \Gamma_{\overline{\alpha}}, \Gamma_{\alpha,P(z)}+ \Gamma_{\overline{\alpha},\overline{P(z)}}\rangle$$ and hence has rank 3. 

Otherwise, if $P(z)=Q(z)L(z)$ splits into an irreducible quadratic and a linear factor in $\Q(\alpha)$, we have that the Picard group is: $$\Pic_\Q (\Tilde{S}) = \langle H, \Gamma_{\alpha}+\Gamma_{\overline{\alpha}}, \Gamma_{\alpha,Q(z)}+\Gamma_{\overline{\alpha},\overline{Q(z)}}, \Gamma_{\alpha,L(z)} + \Gamma_{\overline{\alpha},\overline{L(z)}}\rangle.$$
Thus, we get $\varrho_{A,B,Q} =4$ in this case. 

Finally, if $P(z)$ splits completely, we have that the Picard group $$\Pic_{\Q}(\Tilde{S}) = \langle H, \Gamma_{\alpha}+\Gamma_{\overline{\alpha}}, \Gamma_{\alpha,z-r_1} + \Gamma_{\overline{\alpha}, \overline{z-r_1}},\Gamma_{\alpha,z-r_2} + \Gamma_{\overline{\alpha}, \overline{z-r_2}}, \Gamma_{\alpha,z-r_3} + \Gamma_{\overline{\alpha}, \overline{z-r_3}}\rangle.$$
This Picard group then has rank $\varrho_{A,B,Q}= 5$. In summary, we have that 
\begin{equation}
    \varrho_{A,B,Q} = \begin{cases}
        3, & P(z) \textrm{ is irreducible in }K_Q,\\
        4, & P(z) \textrm{ factorizes into a linear and an irreducible quadratic in }K_Q, \\
        5, & P(z) \textrm{ splits completely in }K_Q.
    \end{cases}
\end{equation}

\section{Examples and lower bounds}\label{sec: ex and lower bounds}

In this section, we compute Theorem \ref{thm: Manin} for certain values of $A$ and $B$ such that $4A^3-27B^2\neq 0$. In particular, we check when $c_{A,B,Q}>0$ (and consequently derive lower bounds towards Manin's conjecture) for some examples.

\subsection{The surface $S_{-1,0,u^2+v^2}$} 
In this case, we look at the surface:
$$y^2 = x^3 - x (u^2+v^2)^2.$$
For this surface, we have that $I = 4$ and $J= 0$. As computed by Duke in \cite[Example 1]{Duke}, there are two classes of binary quartics in $\calF(1,0)$:
$$x^4 + 6x^2y^2 + y^4, \textrm{ and }x^4 + 4y^4.$$

First, we calculate the exponent $\beta_{-1,0,u^2+v^2}.$ Note that $K_Q = \Q(i)$ since $Q(u,v)=u^2+v^2$; since $K_Q$ has class number one, the Hilbert class field $\calH = K_Q$ and the Proposition \ref{prop: cusp} can always be applied. Next, to calculate the exponents we treat each binary quartic separately:
\begin{itemize}
    \item Consider $x^4+6x^2y^2 + y^4 = (x^2+3y^2)^2 - 8 y^4.$ The splitting field of this binary quartic form is given by $$K_F = \Q(\sqrt{2\sqrt{2}-3}, \sqrt{-3-2\sqrt{2}}),$$
    which does not contain $K_Q$. Thus, $\beta_{x^4+6x^2y^2+y^4, u^2+v^2}=0.$
    \item Consider $x^4+4y^4 = (x^2+2xy+2y^2)(x^2-2xy+2y^2).$ The splitting field of this binary quartic form is given by 
    $$K_F = \Q(i) = K_Q.$$
    Thus, $\beta_{x^4+4y^4,u^2+v^2}=2.$
\end{itemize}
As a result, $\beta_{-1,0,u^2+v^2} = 2$ and we know that 
$$N(T) = c_{-1,0,u^2+v^2} \cdot T\log(T)^2 + O(T\log(T)^{2-10^{-8}}).$$

Next, we hope to analyze $c_{-1,0,u^2+v^2}$ and show that it is strictly positive. Since $u^2+v^2$ is a norm form of $\Q(i)$, we can apply \cite[Proposition 8.1, \S12]{MyManinChatelet} to deduce that $c_{-1,0,u^2+v^2}=0$ if and only if the Ch\^atelet surface given by $\{x^2+y^2 = z^4+4\}$ has a local or Brauer-Manin obstruction. One can see that the points $u=0,v=\pm 2, z=0$ are rational points on the surface $u^2+v^2 = z^4+4.$ So, $c_{x^4+4y^4,u^2+v^2}>0$.

Consequently, we deduce a lower bound towards Manin's conjecture. Specifically, let $\tilde{N}(T)$ be defined as in Conjecture \ref{conj: Manin}; then we know that 
$$\tilde{N}(T)\gg T\log(T)^2.$$
In comparison, Manin's conjecture predicts that since $\varrho_{-1,0,u^2+v^2} =5$ that $\tilde{N}(T)\asymp T\log(T)^4$ in this case.

\subsection{The surface $S_{-1,0,u^2+5v^2}$}
In this case, we look at the surface $$y^2 = x^3 - x(u^2+5v^2)^2.$$ Again, we have that $A =-1$ and $B=0$; thus we take $I= 4$ and $J=0$. From our computations of the splitting fields of $x^4+6x^2y^2+y^4, x^4+4y^4$ above, we can see that $(-1,0,u^2+5v^2)$ is a disassociated tuple. Thus, we can apply directly Theorem \ref{thm: Manin} and see that 
$$N(T) = c_{-1,0,u^2+5v^2} \cdot  T + O(T\log(T)^{-10^{-8}}).$$

We know that $u^2+5v^2$ is a norm form for the quadratic field $\Q(\sqrt{-5})$; thus, to determine if $c_{-1,0,u^2+5v^2}>0$ it suffices to study two Ch\^atelet surfaces: 
$$ x^2+5y^2 = z^4+6z^2+1, \hspace{0.5cm} x^2+5y^2 = z^4+4.$$
In fact, if we can find a rational point on either Ch\^atelet surface then we know that $c_{-1,0,u^2+5v^2}>0.$ We can easily see that $(\pm 1, 0, 1)$ (resp. $(0, \pm 1, \pm 1)$) is a rational point on the first (resp. second) Ch\^atelet surface. Define $\tilde{N}(T)$ again as in Conjecture \ref{conj: Manin}; we again deduce a lower bound towards Manin's conjecture: 
$$\tilde{N}(T)\gg T.$$
In comparison, since $z^3+20z$ has two irreducible factors over $K_Q$, we have that $\varrho_{A,B,Q} = 4$. So, Manin's conjecture predicts in this case that $\tilde{N}(T)\asymp T\log(T)^{3}$. 

\subsection{The surface $S_{-97/48, 955/864, u^2+v^2}$} \label{subsec: BM example}Finally, we consider the surface $$y^2 = x^3 - \frac{97}{48} x (u^2+v^2)^2 + \frac{955}{864}(u^2+v^2)^3.$$
In this case, we have that $A = -97/48$ and $B = 955/864$, which are not integers. Nevertheless, we can proceed to look at the integral solutions to this surface given by Mordell's parameterization. Using that $I = 97/12$ and $J = -955/216$ and the fundamental domain described by Mordell \cite{Mordell} in his finiteness theorem for the class number of binary quartics, we achieve that there are four classes in $\calF(97/12, -955/216):$
$$-x^4 + 5x^2y^2 - 6y^4, x^4 + 5x^2y^2 + 6y^4, -2x^4 + 5x^2y^2 - 3y^4, 2x^4 + 5x^2y^2 + 3y^4.$$
Note that all four of the above expression factor further.
In particular,
$$-x^4 + 5x^2y^2 - 6y^4 = (-x^2+3y^2)(x^2-2y^2) \implies K_{-x^4+5x^2y^2-6y^4} = \Q(\sqrt{2},\sqrt{3}).$$
$$x^4+5x^2y^2+6y^4 = (x^2+3y^2)(x^2+2y^2) \implies K_{x^4+5x^2y^2+6y^4} = \Q(\sqrt{-2},\sqrt{-3}).$$
$$-2x^4+5x^2y^2 - 3y^4 = (y-x)(y+x)(2x^2-3y^2) \implies K_{-2x^4+5x^2y^2-3y^4} = \Q(\sqrt{3/2}).$$
$$2x^4+5x^2y^2+3y^4 = (x^2+y^2)(2x^2+3y^2)\implies K_{2x^4+5x^2y^2 + 3y^4} = \Q(i,\sqrt{-3/2}).$$
From the above calculations, we can see that:
$$\beta_{-x^4 + 5x^2y^2 - 6y^4, u^2+v^2}=0,\beta_{x^4+5x^2y^2+6y^4, u^2+v^2}=0,$$ $$\beta_{-2x^4+5x^2y^2-3y^4, u^2+v^2}=0, \beta_{2x^4+5x^2y^2+3y^4,u^2+v^2}=1. $$
Thus, we get that $\beta_{-97/48,955/864,u^2+v^2}= 1.$ 

The reason we focus on this particular example is that it demonstrates an interesting phenomenon: there exists a Brauer-Manin obstruction to rational points in one of the fibers. Specifically, Iskovskikh's famous counter example to the Hasse principle for Ch\^atelet surfaces \cite{Iskovskikh} appears to force $c_{-x^4+5x^2y^2-6y^4, u^2+v^2}=0.$

On the other hand, we can construct explicit points on the Ch\^atelet surface $$u^2+v^2 = 3z^4 + 5z^2 + 2.$$ So, we must have that $c_{2x^4+5x^2y^2+3y^4, u^2+v^2}>0.$ Since we also calculated that $2x^4+5x^2y^2+3y^4$ is the only form for which $\beta_{F,Q}>0$, we must have that on a particular thin subset of $$y^2= x^3 - \frac{97}{48}x (u^2+v^2)^2 + \frac{955}{864}(u^2+v^2)^3,$$
there are $\asymp T\log(T)$
integral points.
\begin{remark}
    This example is a surface where there is a Brauer-Manin obstruction blocking the existence of integral points  corresponding to one of the relevant binary quartic forms, but after ranging over all of the binary quartics in $\calF(97/12, -955/216)$ there will still be integral points on the other fibers. It would be interesting to see if one can construct a surface $S_{A,B,Q}$ where the Brauer-Manin obstructions block the existence of points on all fibers. 
\end{remark}

\bibliographystyle{abbrv}
\bibliography{biblio}

@article{Duke,
    author = {Duke, W},
    title = "{On Elliptic Curves and Binary Quartic Forms}",
    journal = {International Mathematics Research Notices},
    year = {2021},
    month = {09},
    issn = {1073-7928},
    doi = {10.1093/imrn/rnab249},
    url = {https://doi.org/10.1093/imrn/rnab249},
    eprint = {https://academic.oup.com/imrn/advance-article-pdf/doi/10.1093/imrn/rnab249/40354555/rnab249.pdf},
}

@article {LeBoudec-dp1-bound,
    AUTHOR = {Le Boudec, Pierre},
     TITLE = {Linear growth for certain elliptic fibrations},
   JOURNAL = {Int. Math. Res. Not. IMRN},
  FJOURNAL = {International Mathematics Research Notices. IMRN},
      YEAR = {2015},
    NUMBER = {21},
     PAGES = {10859--10871},
      ISSN = {1073-7928},
   MRCLASS = {14G05 (14G25 14J26)},
  MRNUMBER = {3456030},
MRREVIEWER = {Boris \`E. Kunyavski\u{\i}},
       DOI = {10.1093/imrn/rnu251},
       URL = {https://doi.org/10.1093/imrn/rnu251},
}

@article {Munshi-dp1-bound,
    AUTHOR = {Munshi, Ritabrata},
     TITLE = {Density of rational points on elliptic fibrations. {II}},
   JOURNAL = {Acta Arith.},
  FJOURNAL = {Acta Arithmetica},
    VOLUME = {134},
      YEAR = {2008},
    NUMBER = {2},
     PAGES = {133--140},
      ISSN = {0065-1036},
   MRCLASS = {14G05 (11D45 11G35)},
  MRNUMBER = {2429642},
MRREVIEWER = {Timothy D. Browning},
       DOI = {10.4064/aa134-2-4},
       URL = {https://doi.org/10.4064/aa134-2-4},
}

@book {Browning-book,
    AUTHOR = {Browning, Timothy D.},
     TITLE = {Quantitative arithmetic of projective varieties},
    SERIES = {Progress in Mathematics},
    VOLUME = {277},
 PUBLISHER = {Birkh\"{a}user Verlag, Basel},
      YEAR = {2009},
     PAGES = {xiv+160},
      ISBN = {978-3-0346-0128-3},
   MRCLASS = {11-02 (11D45 11G35 14G05 14G25)},
  MRNUMBER = {2559866},
MRREVIEWER = {Robert Juricevic},
       DOI = {10.1007/978-3-0346-0129-0},
       URL = {https://doi.org/10.1007/978-3-0346-0129-0},
}

@article {delaBrowning-Manin-degree-4,
    AUTHOR = {de la Bret\`eche, R. and Browning, T. D.},
     TITLE = {On {M}anin's conjecture for singular del {P}ezzo surfaces of
              degree four. {II}},
   JOURNAL = {Math. Proc. Cambridge Philos. Soc.},
  FJOURNAL = {Mathematical Proceedings of the Cambridge Philosophical
              Society},
    VOLUME = {143},
      YEAR = {2007},
    NUMBER = {3},
     PAGES = {579--605},
      ISSN = {0305-0041},
   MRCLASS = {14G05 (14J26)},
  MRNUMBER = {2373960},
MRREVIEWER = {Yuri Tschinkel},
       DOI = {10.1017/S0305004107000205},
       URL = {https://doi.org/10.1017/S0305004107000205},
}

@article {delaBrowning-Manin-degree-4-2,
    AUTHOR = {de la Bret\`eche, R. and Browning, T. D.},
     TITLE = {On {M}anin's conjecture for singular del {P}ezzo surfaces of
              degree four. {I}},
   JOURNAL = {Michigan Math. J.},
  FJOURNAL = {Michigan Mathematical Journal},
    VOLUME = {55},
      YEAR = {2007},
    NUMBER = {1},
     PAGES = {51--80},
      ISSN = {0026-2285},
   MRCLASS = {14G05 (11G35)},
  MRNUMBER = {2320172},
MRREVIEWER = {Yuri Tschinkel},
       DOI = {10.1307/mmj/1177681985},
       URL = {https://doi.org/10.1307/mmj/1177681985},
}

@incollection {Browning-Survey,
    AUTHOR = {Browning, T. D.},
     TITLE = {An overview of {M}anin's conjecture for del {P}ezzo surfaces},
 BOOKTITLE = {Analytic number theory},
    SERIES = {Clay Math. Proc.},
    VOLUME = {7},
     PAGES = {39--55},
 PUBLISHER = {Amer. Math. Soc., Providence, RI},
      YEAR = {2007},
   MRCLASS = {14G05 (11G35 14G25 14J26)},
  MRNUMBER = {2362193},
MRREVIEWER = {Hizuru Yamagishi},
}

@article{FreiLoughranSofos,
author = {Frei, Christopher and Loughran, Daniel and Sofos, Efthymios},
title = {Rational points of bounded height on general conic bundle surfaces},
journal = {Proceedings of the London Mathematical Society},
volume = {117},
number = {2},
pages = {407-440},
doi = {https://doi.org/10.1112/plms.12134},
url = {https://londmathsoc.onlinelibrary.wiley.com/doi/abs/10.1112/plms.12134},
eprint = {https://londmathsoc.onlinelibrary.wiley.com/doi/pdf/10.1112/plms.12134},
year = {2018}
}

@article {delaTenenbaum-Manin,
    AUTHOR = {de la Bret\`eche, R\'{e}gis and Tenenbaum, G\'{e}rald},
     TITLE = {Sur la conjecture de {M}anin pour certaines surfaces de
              {C}h\^{a}telet},
   JOURNAL = {J. Inst. Math. Jussieu},
  FJOURNAL = {Journal of the Institute of Mathematics of Jussieu. JIMJ.
              Journal de l'Institut de Math\'{e}matiques de Jussieu},
    VOLUME = {12},
      YEAR = {2013},
    NUMBER = {4},
     PAGES = {759--819},
      ISSN = {1474-7480},
   MRCLASS = {11D72 (11D09 11D45 11E12 11G35 11N37 11P55 14G05 14G25)},
  MRNUMBER = {3103132},
MRREVIEWER = {D. R. Heath-Brown},
       DOI = {10.1017/S1474748012000886},
       URL = {https://doi.org/10.1017/S1474748012000886},
}

@incollection {HB-linear,
    AUTHOR = {Heath-Brown, D. R.},
     TITLE = {Linear relations amongst sums of two squares},
 BOOKTITLE = {Number theory and algebraic geometry},
    SERIES = {London Math. Soc. Lecture Note Ser.},
    VOLUME = {303},
     PAGES = {133--176},
 PUBLISHER = {Cambridge Univ. Press, Cambridge},
      YEAR = {2003},
   MRCLASS = {11N25},
  MRNUMBER = {2053459},
MRREVIEWER = {G. Greaves},
}

@article {ChambertLoirTschinkel,
    AUTHOR = {Chambert-Loir, Antoine and Tschinkel, Yuri},
     TITLE = {Igusa integrals and volume asymptotics in analytic and adelic
              geometry},
   JOURNAL = {Confluentes Math.},
  FJOURNAL = {Confluentes Mathematici},
    VOLUME = {2},
      YEAR = {2010},
    NUMBER = {3},
     PAGES = {351--429},
      ISSN = {1793-7442},
   MRCLASS = {11G50 (11G35 14G05)},
  MRNUMBER = {2740045},
MRREVIEWER = {Boris \`E. Kunyavski\u{\i}},
       DOI = {10.1142/S1793744210000223},
       URL = {https://doi.org/10.1142/S1793744210000223},
}

@article {Cayley,
    AUTHOR = {Cayley, A.},
     TITLE = {Note sur les covariants d'une fonction qudratique, cubique, ou
              biquadratique \`a deux ind\'{e}termin\'{e}es},
   JOURNAL = {J. Reine Angew. Math.},
  FJOURNAL = {Journal f\"{u}r die Reine und Angewandte Mathematik. [Crelle's
              Journal]},
    VOLUME = {50},
      YEAR = {1855},
     PAGES = {285--287},
      ISSN = {0075-4102},
   MRCLASS = {DML},
  MRNUMBER = {1578947},
       DOI = {10.1515/crll.1855.50.285},
       URL = {https://doi.org/10.1515/crll.1855.50.285},
}

@article {Hermite,
    AUTHOR = {Hermite, C.},
     TITLE = {Note sur la r\'{e}duction des fonctions homog\`enes \`a coefficients
              entiers et \`a deux ind\'{e}termin\'{e}es},
   JOURNAL = {J. Reine Angew. Math.},
  FJOURNAL = {Journal f\"{u}r die Reine und Angewandte Mathematik. [Crelle's
              Journal]},
    VOLUME = {36},
      YEAR = {1848},
     PAGES = {357--364},
      ISSN = {0075-4102},
   MRCLASS = {DML},
  MRNUMBER = {1578622},
       DOI = {10.1515/crll.1848.36.357},
       URL = {https://doi.org/10.1515/crll.1848.36.357},
}

@article{Mordell,
  title={Indeterminate equations of the third and fourth degrees},
  author={Mordell, LJ},
  journal={Quart. J. pure appl. Math},
  volume={45},
  pages={170--186},
  year={1914}
}

@article {delaBrowning-linear,
    AUTHOR = {de la Bret\`eche, R. and Browning, T. D.},
     TITLE = {Binary linear forms as sums of two squares},
   JOURNAL = {Compos. Math.},
  FJOURNAL = {Compositio Mathematica},
    VOLUME = {144},
      YEAR = {2008},
    NUMBER = {6},
     PAGES = {1375--1402},
      ISSN = {0010-437X},
   MRCLASS = {11N37 (11E16 11N25)},
  MRNUMBER = {2474314},
MRREVIEWER = {R. A. Mollin},
       DOI = {10.1112/S0010437X08003692},
       URL = {https://doi.org/10.1112/S0010437X08003692},
}

@article {delaBrowning-quad,
    AUTHOR = {de la Bret\`eche, R. and Browning, T. D.},
     TITLE = {Le probl\`eme des diviseurs pour des formes binaires de degr\'{e} 4},
   JOURNAL = {J. Reine Angew. Math.},
  FJOURNAL = {Journal f\"{u}r die Reine und Angewandte Mathematik. [Crelle's
              Journal]},
    VOLUME = {646},
      YEAR = {2010},
     PAGES = {1--44},
      ISSN = {0075-4102},
   MRCLASS = {11N32 (11D45 11N56)},
  MRNUMBER = {2719554},
MRREVIEWER = {D. R. Heath-Brown},
       DOI = {10.1515/CRELLE.2010.064},
       URL = {https://doi.org/10.1515/CRELLE.2010.064},
}

@article {Derenthal,
    AUTHOR = {Derenthal, Ulrich},
     TITLE = {Singular del {P}ezzo surfaces whose universal torsors are
              hypersurfaces},
   JOURNAL = {Proc. Lond. Math. Soc. (3)},
  FJOURNAL = {Proceedings of the London Mathematical Society. Third Series},
    VOLUME = {108},
      YEAR = {2014},
    NUMBER = {3},
     PAGES = {638--681},
      ISSN = {0024-6115},
   MRCLASS = {14J26 (14C20 14G05 14J17)},
  MRNUMBER = {3180592},
MRREVIEWER = {Antonio Laface},
       DOI = {10.1112/plms/pdt041},
       URL = {https://doi.org/10.1112/plms/pdt041},
}

@article {Matthiesen,
    AUTHOR = {Matthiesen, Lilian},
     TITLE = {Correlations of the divisor function},
   JOURNAL = {Proc. Lond. Math. Soc. (3)},
  FJOURNAL = {Proceedings of the London Mathematical Society. Third Series},
    VOLUME = {104},
      YEAR = {2012},
    NUMBER = {4},
     PAGES = {827--858},
      ISSN = {0024-6115},
   MRCLASS = {11N37 (11B30 11N64)},
  MRNUMBER = {2908784},
MRREVIEWER = {Julia Wolf},
       DOI = {10.1112/plms/pdr046},
       URL = {https://doi.org/10.1112/plms/pdr046},
}

@article {BhargavaShankar,
    AUTHOR = {Bhargava, Manjul and Shankar, Arul},
     TITLE = {Binary quartic forms having bounded invariants, and the
              boundedness of the average rank of elliptic curves},
   JOURNAL = {Ann. of Math. (2)},
  FJOURNAL = {Annals of Mathematics. Second Series},
    VOLUME = {181},
      YEAR = {2015},
    NUMBER = {1},
     PAGES = {191--242},
      ISSN = {0003-486X},
   MRCLASS = {11E20 (14G25)},
  MRNUMBER = {3272925},
MRREVIEWER = {John M. Voight},
       DOI = {10.4007/annals.2015.181.1.3},
       URL = {https://doi.org/10.4007/annals.2015.181.1.3},
}

@article {Baier-Browning-delPezzodegree2,
    AUTHOR = {Baier, Stephan and Browning, Tim D.},
     TITLE = {Inhomogeneous cubic congruences and rational points on del
              {P}ezzo surfaces},
   JOURNAL = {J. Reine Angew. Math.},
  FJOURNAL = {Journal f\"{u}r die Reine und Angewandte Mathematik. [Crelle's
              Journal]},
    VOLUME = {680},
      YEAR = {2013},
     PAGES = {69--151},
      ISSN = {0075-4102},
   MRCLASS = {11G05 (14G15 14J26)},
  MRNUMBER = {3100953},
MRREVIEWER = {Ulrich Derenthal},
       DOI = {10.1515/crelle.2012.039},
       URL = {https://doi.org/10.1515/crelle.2012.039},
}

@article {ECrankBhargavaetal,
    AUTHOR = {Bhargava, M. and Shankar, A. and Taniguchi, T. and Thorne, F.
              and Tsimerman, J. and Zhao, Y.},
     TITLE = {Bounds on 2-torsion in class groups of number fields and
              integral points on elliptic curves},
   JOURNAL = {J. Amer. Math. Soc.},
  FJOURNAL = {Journal of the American Mathematical Society},
    VOLUME = {33},
      YEAR = {2020},
    NUMBER = {4},
     PAGES = {1087--1099},
      ISSN = {0894-0347},
   MRCLASS = {11N45 (11G05 11R29)},
  MRNUMBER = {4155220},
MRREVIEWER = {Olivier Bordell\`es},
       DOI = {10.1090/jams/945},
       URL = {https://doi.org/10.1090/jams/945},
}

@article {ECrankHelfgottVenkatesh,
    AUTHOR = {Helfgott, H. A. and Venkatesh, A.},
     TITLE = {Integral points on elliptic curves and 3-torsion in class
              groups},
   JOURNAL = {J. Amer. Math. Soc.},
  FJOURNAL = {Journal of the American Mathematical Society},
    VOLUME = {19},
      YEAR = {2006},
    NUMBER = {3},
     PAGES = {527--550},
      ISSN = {0894-0347},
   MRCLASS = {11G05 (11R11 11R29 14G05)},
  MRNUMBER = {2220098},
MRREVIEWER = {\'{A}lvaro Lozano-Robledo},
       DOI = {10.1090/S0894-0347-06-00515-7},
       URL = {https://doi.org/10.1090/S0894-0347-06-00515-7},
}

@book {IwaniecKowalski,
    AUTHOR = {Iwaniec, Henryk and Kowalski, Emmanuel},
     TITLE = {Analytic number theory},
    SERIES = {American Mathematical Society Colloquium Publications},
    VOLUME = {53},
 PUBLISHER = {American Mathematical Society, Providence, RI},
      YEAR = {2004},
     PAGES = {xii+615},
      ISBN = {0-8218-3633-1},
   MRCLASS = {11-02 (11Fxx 11Lxx 11Mxx 11Nxx)},
  MRNUMBER = {2061214},
MRREVIEWER = {K. Soundararajan},
       DOI = {10.1090/coll/053},
       URL = {https://doi.org/10.1090/coll/053},
}

@incollection {Swinnerton-Dyer,
    AUTHOR = {Swinnerton-Dyer, Peter},
     TITLE = {Diophantine equations: progress and problems},
 BOOKTITLE = {Arithmetic of higher-dimensional algebraic varieties ({P}alo
              {A}lto, {CA}, 2002)},
    SERIES = {Progr. Math.},
    VOLUME = {226},
     PAGES = {3--35},
 PUBLISHER = {Birkh\"{a}user Boston, Boston, MA},
      YEAR = {2004},
   MRCLASS = {11G35 (11D72 14G05)},
  MRNUMBER = {2028898},
MRREVIEWER = {Timothy D. Browning},
       DOI = {10.1007/978-0-8176-8170-8\_1},
       URL = {https://doi.org/10.1007/978-0-8176-8170-8_1},
}

@article {PeyreConstant,
    AUTHOR = {Peyre, Emmanuel},
     TITLE = {Hauteurs et mesures de {T}amagawa sur les vari\'{e}t\'{e}s de {F}ano},
   JOURNAL = {Duke Math. J.},
  FJOURNAL = {Duke Mathematical Journal},
    VOLUME = {79},
      YEAR = {1995},
    NUMBER = {1},
     PAGES = {101--218},
      ISSN = {0012-7094},
   MRCLASS = {11G35 (14G05 14J45)},
  MRNUMBER = {1340296},
MRREVIEWER = {Shouwu Zhang},
       DOI = {10.1215/S0012-7094-95-07904-6},
       URL = {https://doi.org/10.1215/S0012-7094-95-07904-6},
}

@article {BatyrevManin,
    AUTHOR = {Batyrev, V. V. and Manin, Yu. I.},
     TITLE = {Sur le nombre des points rationnels de hauteur born\'{e} des
              vari\'{e}t\'{e}s alg\'{e}briques},
   JOURNAL = {Math. Ann.},
  FJOURNAL = {Mathematische Annalen},
    VOLUME = {286},
      YEAR = {1990},
    NUMBER = {1-3},
     PAGES = {27--43},
      ISSN = {0025-5831},
   MRCLASS = {11G35 (14G10 14G40)},
  MRNUMBER = {1032922},
MRREVIEWER = {Marc Hindry},
       DOI = {10.1007/BF01453564},
       URL = {https://doi.org/10.1007/BF01453564},
}

@article {Destagnol,
    AUTHOR = {Destagnol, Kevin},
     TITLE = {La conjecture de {M}anin pour certaines surfaces de
              {C}h\^{a}telet},
   JOURNAL = {Acta Arith.},
  FJOURNAL = {Acta Arithmetica},
    VOLUME = {174},
      YEAR = {2016},
    NUMBER = {1},
     PAGES = {31--97},
      ISSN = {0065-1036},
   MRCLASS = {11D45 (11D57 11N37)},
  MRNUMBER = {3517531},
MRREVIEWER = {Abderrahmane Nitaj},
       DOI = {10.4064/aa8312-2-2016},
       URL = {https://doi.org/10.4064/aa8312-2-2016},
}

@article{DerenthalWilsch,
   title={INTEGRAL POINTS ON SINGULAR DEL PEZZO SURFACES},
   ISSN={1475-3030},
   url={http://dx.doi.org/10.1017/S1474748022000482},
   DOI={10.1017/s1474748022000482},
   journal={Journal of the Institute of Mathematics of Jussieu},
   publisher={Cambridge University Press (CUP)},
   author={Derenthal, Ulrich and Wilsch, Florian},
   year={2022},
   month=nov, pages={1–36} }

@ARTICLE{MyManinChatelet,
       author = {{Woo}, Katharine},
        title = "{On Manin's conjecture for Ch{\^a}telet surfaces}",
      journal = {arXiv e-prints},
         year = 2024,
        month = sep,
          eid = {arXiv:2409.17381},
        pages = {arXiv:2409.17381},
          doi = {10.48550/arXiv.2409.17381},
archivePrefix = {arXiv},
       eprint = {2409.17381},
 primaryClass = {math.NT},
       adsurl = {https://ui.adsabs.harvard.edu/abs/2024arXiv240917381W}
}

@article {Iskovskikh,
    AUTHOR = {Iskovskih, V. A.},
     TITLE = {A counterexample to the {H}asse principle for systems of two
              quadratic forms in five variables},
   JOURNAL = {Mat. Zametki},
  FJOURNAL = {Akademiya Nauk SSSR. Matematicheskie Zametki},
    VOLUME = {10},
      YEAR = {1971},
     PAGES = {253--257},
      ISSN = {0025-567X},
   MRCLASS = {10.10 (14.00)},
  MRNUMBER = {286743},
MRREVIEWER = {G. Maxwell},
}

@book {SerreBook,
    AUTHOR = {Serre, J.-P.},
     TITLE = {Lectures on the {M}ordell-{W}eil theorem},
    SERIES = {Aspects of Mathematics},
   EDITION = {3rd},
      NOTE = {Translated from the French and edited by M. Brown from
              notes by M. Waldschmidt.},
 PUBLISHER = {Friedr. Vieweg \& Sohn, Braunschweig},
      YEAR = {1997},
     PAGES = {x+218},
      ISBN = {3-528-28968-6},
   MRCLASS = {11G10 (11D41 11G30 14G25)},
  MRNUMBER = {1757192},
       DOI = {10.1007/978-3-663-10632-6},
       URL = {https://doi.org/10.1007/978-3-663-10632-6},
}

@incollection {HBCubic,
    AUTHOR = {Heath-Brown, Roger},
     TITLE = {Counting rational points on cubic surfaces},
      NOTE = {Nombre et r\'{e}partition de points de hauteur born\'{e}e (Paris,
              1996)},
   JOURNAL = {Ast\'{e}risque},
  FJOURNAL = {Ast\'{e}risque},
    NUMBER = {251},
      YEAR = {1998},
     PAGES = {13--30},
      ISSN = {0303-1179},
   MRCLASS = {11D45 (11G05 11G35 14G05)},
  MRNUMBER = {1679837},
MRREVIEWER = {Joseph H. Silverman},
}

@article {BonolisBrowning,
    AUTHOR = {Bonolis, Dante and Browning, Tim},
     TITLE = {Uniform bounds for rational points on hyperelliptic
              fibrations},
   JOURNAL = {Ann. Sc. Norm. Super. Pisa Cl. Sci. (5)},
  FJOURNAL = {Annali della Scuola Normale Superiore di Pisa. Classe di
              Scienze. Serie V},
    VOLUME = {24},
      YEAR = {2023},
    NUMBER = {1},
     PAGES = {173--204},
      ISSN = {0391-173X},
   MRCLASS = {11D45 (11L40 11N36 14G05)},
  MRNUMBER = {4587744},
MRREVIEWER = {Harald A. Helfgott},
}

@ARTICLE{Santens,
       author = {{Santens}, Tim},
        title = "{Manin's conjecture for integral points on toric varieties}",
      journal = {arXiv e-prints},
         year = 2023,
        month = dec,
          eid = {arXiv:2312.13914},
        pages = {arXiv:2312.13914},
          doi = {10.48550/arXiv.2312.13914},
archivePrefix = {arXiv},
       eprint = {2312.13914},
 primaryClass = {math.NT},
       adsurl = {https://ui.adsabs.harvard.edu/abs/2023arXiv231213914S}
}

@inproceedings {GerardinLabesseBaseChange,
    AUTHOR = {G\'{e}rardin, P. and Labesse, J.-P.},
     TITLE = {The solution of a base change problem for {${\rm GL}(2)$}
              (following {L}anglands, {S}aito, {S}hintani)},
 BOOKTITLE = {Automorphic forms, representations and {$L$}-functions
              ({P}roc. {S}ympos. {P}ure {M}ath., {O}regon {S}tate {U}niv.,
              {C}orvallis, {O}re., 1977), {P}art 2},
    SERIES = {Proc. Sympos. Pure Math., XXXIII},
     PAGES = {115--133},
 PUBLISHER = {Amer. Math. Soc., Providence, RI},
      YEAR = {1979},
   MRCLASS = {10D40 (22E55)},
  MRNUMBER = {546613},
MRREVIEWER = {Stephen Gelbart},
}

@article {delaTenenbaumSieve,
    AUTHOR = {de la Bret\`eche, R\'{e}gis and Tenenbaum, G\'{e}rald},
     TITLE = {Moyennes de fonctions arithm\'{e}tiques de formes binaires},
   JOURNAL = {Mathematika},
  FJOURNAL = {Mathematika. A Journal of Pure and Applied Mathematics},
    VOLUME = {58},
      YEAR = {2012},
    NUMBER = {2},
     PAGES = {290--304},
      ISSN = {0025-5793},
   MRCLASS = {11A25 (11D45 11N37 11N56)},
  MRNUMBER = {2965973},
MRREVIEWER = {Olivier Bordell\`es},
       DOI = {10.1112/S0025579311002154},
       URL = {https://doi.org/10.1112/S0025579311002154},
}

@article {Daniel,
    AUTHOR = {Daniel, Stephan},
     TITLE = {On the divisor-sum problem for binary forms},
   JOURNAL = {J. Reine Angew. Math.},
  FJOURNAL = {Journal f\"{u}r die Reine und Angewandte Mathematik. [Crelle's
              Journal]},
    VOLUME = {507},
      YEAR = {1999},
     PAGES = {107--129},
      ISSN = {0075-4102},
   MRCLASS = {11N37},
  MRNUMBER = {1670278},
MRREVIEWER = {G. Greaves},
       DOI = {10.1515/crll.1999.010},
       URL = {https://doi.org/10.1515/crll.1999.010},
}

@article{MyBaseChange,
    author = {Woo, Katharine},
    title = {Base change for {GL(2)} and sums of Hecke eigenvalues along polynomial sequences},
    journal = preprint,
    year = 2026
}

\end{document}